\documentclass[reqno]{amsart}
\usepackage[sc]{mathpazo}

\usepackage{amsmath}
\usepackage{amsthm}
\usepackage{amsfonts}
\usepackage{graphicx,psfrag,epsfig}
\usepackage{color}

\usepackage{pdfsync}

\newtheorem{thm}{Theorem}[section]
\newtheorem{lem}[thm]{Lemma}
\newtheorem{cor}[thm]{Corollary}
\newtheorem{prop}[thm]{Proposition}

\DeclareMathAlphabet{\mathpzc}{OT1}{pzc}{m}{it}

\numberwithin{equation}{section}

\newcommand{\bqn}{\begin{equation}}
\newcommand{\eqn}{\end{equation}}

\newcommand{\R}{\mathbb{R}}

\newcommand{\ve}{\varepsilon}
\newcommand{\e}{\varepsilon}
\newcommand{\rd}{\mathrm{d}}

\newcommand{\ue}{u_\varepsilon}

\newcommand{\Phie}{\Phi_\varepsilon}

\newcommand{\dhr}{\mathrel{\lhook\joinrel\relbar\kern-.8ex\joinrel\lhook\joinrel\rightarrow}} 

\title[A parabolic free boundary problem for MEMS]
{A parabolic free boundary problem modeling electrostatic MEMS}

\begin{document}
\author{Joachim Escher}
\address{Leibniz Universit\"at Hannover\\ Institut f\" ur Angewandte Mathematik \\ Welfengarten 1 \\ D--30167 Hannover\\ Germany}
\email{escher@ifam.uni-hannover.de}

\author{Philippe Lauren\c{c}ot}
\address{Institut de Math\'ematiques de Toulouse, CNRS UMR~5219, Universit\'e de Toulouse \\ F--31062 Toulouse Cedex 9, France}
\email{laurenco@math.univ-toulouse.fr}

\author{Christoph Walker}
\address{Leibniz Universit\"at Hannover\\ Institut f\" ur Angewandte Mathematik \\ Welfengarten 1 \\ D--30167 Hannover\\ Germany}
\email{walker@ifam.uni-hannover.de}

\begin{abstract}
The evolution problem for a membrane based model of an electrostatically actuated microelectromechanical system (MEMS) is studied. The model describes the dynamics of the membrane displacement and the electric potential. The latter is a harmonic function in an angular domain, the deformable membrane being a part of the boundary. The former solves a heat equation with a right hand side that depends on the square of the trace of the gradient of the electric potential on the membrane. The resulting free boundary problem is shown to be well-posed locally in time. Furthermore, solutions corresponding to small voltage values exist globally in time while global existence is shown not to hold for high voltage values. It is also proven that, for small voltage values, there is an asymptotically stable steady-state solution. Finally, the small aspect ratio limit is rigorously justified.
\end{abstract}


\maketitle

\section{Introduction}

An idealized electrostatically actuated microelectromechanical system (MEMS) consists of a rigid ground plate above which a thin and deformable elastic membrane is suspended that is held fixed along its boundary, see Figure~\ref{fig1}. Applying a voltage difference between the two components induces displacements of the membrane and thus transforms electrostatic energy into mechanical energy, a feature that has applications in the design of transistors, switches, or micro-pumps, for instance. There is, however, an upper limit for the applied voltage potential beyond which the electrostatic force cannot be balanced by the elastic response of the membrane which then touches down on the rigid plate. This phenomenon is usually referred to as ``pull-in'' instability. Estimating this threshold value is an important issue in applications as it may be a desirable feature of the device in some situations (e.g. switches, micropumps) or possibly damage the device in others. Mathematical models have been set up for that purpose, and we refer the reader e.g. to \cite{PeleskoSIAP02, BernsteinPelesko, PeleskoTriolo01} and the references therein for a more detailed account of the physical background and the modeling aspects of such devices. 

\begin{figure}
\centering\includegraphics[width=10cm]{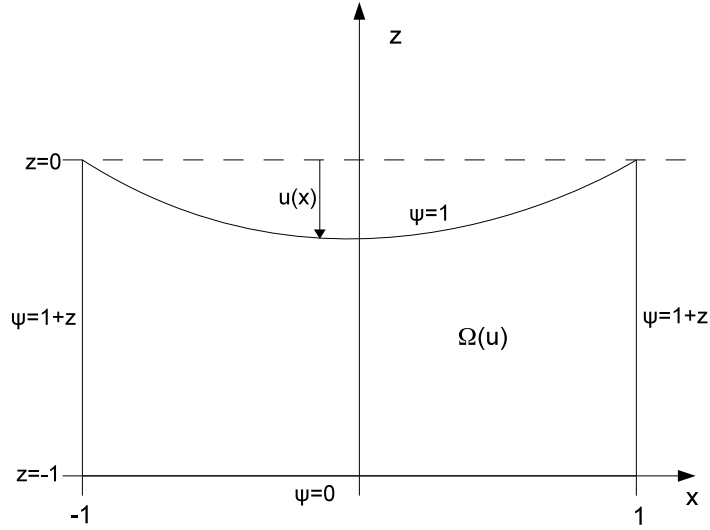}
\caption{\small Idealized electrostatic MEMS device.}\label{fig1}
\end{figure}

Denoting the displacement of the membrane and the electrostatic potential in the device by $u$ and $\psi$, respectively, we consider here the idealized situation where the applied voltage and the permittivity of the membrane are constant (normalized to one), and there is no variation in the horizontal direction orthogonal to the $x$-direction of both $\psi$ and $u$. Under appropriate scalings, the rigid ground plate is at $z=-1$, and the undeflected membrane at $z=0$ is fixed at the boundary $x=-1$ and $x=1$ of the interval $I:=(-1,1)$, see Figure~\ref{fig1}. Letting $\varepsilon$ denote the aspect ratio of the device before scaling, i.e. the ratio of the undeformed gap size to the device length, the membrane displacement $u=u(t,x)\in (-1,\infty)$ evolves according to
\begin{equation}\label{u}
\partial_t u - \partial_x^2 u = - \lambda\ \left( \varepsilon^2\ |\partial_x\psi(t,x,u)|^2 + |\partial_z\psi(t,x,u)|^2 \right)\ ,\quad x\in I\ ,\qquad t>0\ ,
\end{equation}
with clamped boundary conditions
\begin{equation}\label{bcu}
u(t,\pm 1)=0\ ,\quad t>0\ ,
\end{equation}
and initial condition
\begin{equation}\label{ic}
u(0,x)=u^0(x)\ ,\quad x\in I\ .
\end{equation}
The dimensionless electrostatic potential $\psi=\psi(t,x,z)$ satisfies Laplace's equation
\begin{equation}\label{psi}
\varepsilon^2\partial_x^2\psi + \partial_z^2\psi =0\ ,\quad (x,z)\in \Omega(u(t))\ ,\quad t>0\ ,
\end{equation}
in the region
$$
\Omega(u(t)) := \left\{ (x,z)\in I\times (-1,\infty)\ :\ -1 < z < u(t,x) \right\}\,
$$
between the rigid ground plate at $z=-1$ and the deflected membrane.
The boundary conditions for $\psi$ are then
\begin{equation}\label{bcpsi}
\psi(t,x,z)=\frac{1+z}{1+u(t,x)}\ ,\quad (x,z)\in  \partial\Omega(u(t))\ ,\quad t>0\ .
\end{equation}

Equation~\eqref{u} corresponds to the situation in which viscous forces dominate over inertial forces in the system, e.g. see \cite{FloresMercadoPelesko, BernsteinPelesko}. Also, deformations due to bending are neglected in \eqref{u}. Of particular importance in the model is the parameter $\lambda> 0$ which characterizes the relative strengths of electrostatic and mechanical forces and is proportional to the applied voltage. According to the above discussion, the pull-in instability is expected to take place for $\lambda$ large enough.

The analysis of \eqref{u}-\eqref{bcpsi} turns out to be rather complex since \eqref{psi} is a free boundary problem: indeed, the domain between the rigid ground plate and the elastic membrane changes with time. Due to this, equations \eqref{u} and \eqref{psi} are strongly coupled. However, a common assumption made in mathematical analysis hitherto is a vanishing aspect ratio $\varepsilon$ that reduces the free boundary problem to a heat equation with a right hand side involving a singularity when the membrane touches down on the ground plate. More precisely, setting $\varepsilon=0$ allows one to solve  \eqref{psi}-\eqref{bcpsi} explicitly for the potential $\psi=\psi_0$, that is,
\bqn\label{psi0}
\psi_0(t,x,z)=\frac{1+z}{1+u_0(t,x)}\ ,\quad (t,x,z)\in [0,\infty)\times I\times(-1,0)\ ,
\eqn
where the displacement $u=u_0$ now satisfies the so-called small aspect ratio model 
\begin{equation}\label{z0}
\begin{array}{rlll}
\partial_t u_0 - \partial_x^2 u_0 &\!\!\!=\!\!\!& - \displaystyle{\frac{\lambda}{(1+u_0)^2}}\,,\quad & x\in I\,, \quad t\in (0,\infty)\,, \\
u_0(t,\pm1) &\!\!\!=\!\!\!& 0\,, & t\in (0,\infty)\,, \\
u_0(0,x)&\!\!\!=\!\!\!& u^0(x)\,, & x\in I\, .
\end{array}
\end{equation}
Several mathematical results have been obtained for \eqref{z0}, including a characterization of the critical value of $\lambda$ which corresponds to the value beyond which no steady-state exists as well as a possible space dependence of the permittivity of the membrane, see, e.g.,  \cite{EspositoGhoussoubGuo, GhoussoubGuoSIMA07, LinYang, BernsteinPelesko} for the stationary problem and \cite{EspositoGhoussoubGuo, FloresMercadoPeleskoSmyth_SIAP07, GhoussoubGuoNoDEA08, GuoHuWang09, GuoJDE08, GuoJDE08_II, Hui11, BernsteinPelesko} for the evolution problem. Inertial effects are taken into account in \cite{GuoSIAMDynamic10, KavallarisLaceyNikolopoulos}.

To the best of our knowledge, the first analytical research without assumption of a small aspect ratio and thus dedicated to the original free boundary problem \eqref{u}-\eqref{bcpsi} is \cite{LaurencotWalker_ARMA}, where the existence of steady-states has been established for small voltage values $\lambda$ and a non-existence result for steady-states is obtained for large values of $\lambda$.

Here, we address the evolution problem. A rough summary of our results reads as follows: We prove the local well-posedness of \eqref{u}-\eqref{bcpsi} for all voltage values and show that the solutions exist globally in time provided the voltage value is sufficiently small. In contrast to the stationary case \cite{LaurencotWalker_ARMA} it turns out that a $W_\infty^2(I)$-setting is no longer suitable for the $u$-component of \eqref{u}-\eqref{bcpsi}. This is due to the fact that the heat semigroup does not enjoy suitable properties in $L_\infty(I)$. Instead, we are therefore lead to work in the framework of $W_q^2(I)$-spaces for $q<\infty$, which generates additional difficulties as now $\partial_x^2u$ may become unbounded.
For small voltage values we further prove that there is a locally asymptotically stable steady-state. For high voltage values we prove that global existence of solutions does not hold. In addition, we analyze the behavior of the solutions as the small aspect ratio $\varepsilon\to 0$, showing convergence towards \eqref{z0} as expected from a formal analysis. \\

To state precisely our results we introduce for $q\in [2,\infty)$ and $\kappa\in (0,1)$ the set
$$
S_q(\kappa):=\left\{u\in W_{q,D}^2(I)\,;\, \|u\|_{W_{q,D}^2(I)}< 1/\kappa \;\;\text{ and }\;\; -1+\kappa< u(x) \text{ for } x\in I \right\}\ ,
$$
where $ W_{q,D}^{2\alpha}(I):=\{u\in  W_{q}^{2\alpha}(I)\,;\, u(\pm 1)=0\}$ for $2\alpha \in (1/q,2]$ and $ W_{q,D}^{2\alpha}(I):= W_{q}^{2\alpha}(I)$ for $0\le 2\alpha < 1/q$. The local existence result now reads:

\begin{thm}[{\bf Local Well-Posedness}]\label{A}
Let $q\in (2,\infty)$, $\varepsilon>0$, and consider an initial value $u^0\in  W_{q,D}^2(I)$ such that $u^0(x)>-1$ for $x\in I$. Then, the following are true:

\begin{itemize}

\item[(i)] For each voltage value $\lambda>0$, there is a unique maximal solution $(u,\psi)$ to \eqref{u}-\eqref{bcpsi} on the maximal interval of existence $[0,T_m^\varepsilon)$ in the sense that
$$u\in C^1\big([0,T_m^\varepsilon),L_q(I)\big)\cap C\big([0,T_m^\varepsilon), W_{q,D}^2(I)\big)
$$
satisfies \eqref{u}-\eqref{ic} together with
$$
u(t,x)>-1\ ,\quad (t,x)\in [0,T_m^\varepsilon)\times I\ , 
$$ 
and $\psi(t)\in W_q^2\big(\Omega(u(t))\big)$ solves \eqref{psi}-\eqref{bcpsi} on $\Omega(u(t))$ for each $t\in [0,T_m^\varepsilon)$. 

\item[(ii)] If for each $\tau>0$ there is $\kappa(\tau)\in (0,1)$ such that $u(t)\in S_q(\kappa(\tau))$ for $t\in [0,T_m^\varepsilon)\cap [0,\tau]$, then the solution exists globally, that is, $T_m^\varepsilon=\infty$.

\item[(iii)] If $u^0(x)\le 0$ for $x\in I$, then $u(t,x)\le 0$ for $(t,x)\in [0,T_m^\varepsilon)\times I$. If $u^0=u^0(x)$ is even with respect to $x\in I$, then, {for all $t\in [0,T_m^\varepsilon)$}, $u=u(t,x)$ and $\psi=\psi(t,x,z)$ are even with respect to $x\in I$ as well.

\end{itemize}
\end{thm}

The proof of Theorem~\ref{A} is performed as follows. We first transform the Laplace equation \eqref{psi} to a fixed rectangle which results in an elliptic boundary value problem with non-constant coefficients depending on $u$ and its derivatives up to order 2. Solving this elliptic equation (for a given $u$) allows us to interpret the full free boundary problem as a nonlocal semilinear heat equation for $u$ (see \eqref{33}). We then employ
a fixed point argument to solve this evolution problem. Since the nonlinearity in the $u$-equation depends on the trace of the gradient of the potential, two ingredients are essential: precise estimates based on the regularizing effects of the heat semigroup and elaborated investigations of the properties  of the solution to the transformed elliptic problem. The proof is given in Section~\ref{Sec2}. \\

We now address global existence issues. From a physical viewpoint a ``pull-in'' instability occurs for high voltage values. Accordingly, for large values of $\lambda$ solutions cease to exist globally while solutions corresponding to small $\lambda$ values exist globally in time. More precisely, we have:

{
\begin{thm}[{\bf Global Existence}]\label{TD}
Let $q\in (2,\infty)$, $\varepsilon>0$, $\lambda>0$, and let $u^0\in  W_{q,D}^2(I)$ satisfy $-1<u^0(x)\le 0$ for $x\in I$. Let $(u,\psi)$ be the corresponding solution to \eqref{u}-\eqref{bcpsi} on the maximal interval of existence $[0,T_m^\ve)$.
\begin{itemize}
\item[(i)]
Given $\kappa\in (0,1)$ there exists $\lambda_*:=\lambda_*(\kappa,\varepsilon)>0$ and $\kappa_0:=\kappa_0(\kappa,\varepsilon)>0$ such that $T_m^\varepsilon=\infty$ and $u(t)\in S_q(\kappa_0)$ for $t\ge 0$ provided that $u^0\in S_q(\kappa)$ and $\lambda\in (0, \lambda_*)$.

\item[(ii)] 
There is $\lambda^*(\varepsilon)>0$ depending only on $\varepsilon$ such that  $T_m^\varepsilon<\infty$ provided $\lambda>\lambda^*(\ve)$.
\end{itemize}
\end{thm}
}

Note that part (i) of Theorem~\ref{TD} provides uniform estimates on $u$ in the $W_{q}^2(I)$-norm and ensures that $u$ never touches down on -1, not even in infinite time. Its proof is contained in Section~\ref{Sec2} and it is a consequence of the above mentioned fixed point argument. The second part of Theorem~\ref{TD} is proven in Section~\ref{Sec3a} by constructing a suitable strict Lyapunov functional. Let us mention that similar results as stated in Theorem~\ref{TD} are known to hold for the small aspect ratio model \eqref{z0}, see \cite{FloresMercadoPeleskoSmyth_SIAP07, GhoussoubGuoNoDEA08}. However, the nonlocal features of \eqref{u}-\eqref{bcpsi} prevents one from using similar techniques and we thus have to develop an alternative approach. 
Also, there is a qualitative difference of the interpretation of the finiteness of $T_m^\ve$  in Theorem~\ref{TD}(ii). Indeed, according to Theorem~\ref{A}, $T_m^\ve<\infty$ implies that the $W_q^2(I)$-norm of $u$ blows up  or $u$ touches down on $-1$ in finite time. This is in clear contrast to the small aspect ratio model \eqref{z0} for which touchdown is the only mechanism for a finite time singularity. The difference stems from the fact that in \eqref{z0} the nonlinearity is of zero order while for the free boundary problem \eqref{u}-\eqref{bcpsi} the nonlocal nonlinearity is rather of order ``3/2'' in the $L_q$-sense (see Proposition~\ref{L1}). Nevertheless, we strongly believe that finite time touchdown occurs in the present model as well when $T_m^\ve$ is finite.

\medskip

{
We next turn to stability of steady-states. This is a delicate issue since it is expected in analogy to what is known for the small aspect ratio model \cite{EspositoGhoussoubGuo, BernsteinPelesko} that there are two steady states for small $\lambda$ values.  In \cite{LaurencotWalker_ARMA} it was shown that there is at least one steady-state to  \eqref{u}-\eqref{bcpsi} for small values of $\lambda$ (and none for large $\lambda$). We shall refine this result here and prove that, provided $\lambda$ is small, this steady-state is unique with a first component in the set $S_q(\kappa)$ and locally asymptotically stable. 

\begin{thm}[{\bf Asymptotic Stability}]\label{TStable}
Let $q\in (2,\infty)$, $\ve>0$, and $\kappa\in (0,1)$. 
\begin{itemize}
\item[(i)] There are $\delta=\delta(\kappa)>0$ and an analytic function $[\lambda\mapsto U_\lambda]:[0,\delta)\rightarrow W_{q,D}^2(I)$ such that $(U_\lambda,\Psi_\lambda)$ is for each $\lambda\in (0,\delta)$ the unique steady-state to
\eqref{u}-\eqref{bcpsi} with $U_\lambda\in S_q(\kappa)$ and $\Psi_\lambda\in W_2^2(\Omega(U_\lambda))$. Moreover, $U_\lambda$ is negative, convex, and even for $\lambda\in (0,\delta)$ and $U_0=0$.

\item[(ii)] Let $\lambda\in (0,\delta)$. There are $\omega_0,r,R>0$ such that for each initial value $u^0\in W_{q,D}^2(I)$ with $\|u^0-U_\lambda\|_{W_{q,D}^2} <r$, the solution $(u,\psi)$ to \eqref{u}-\eqref{bcpsi} exists globally in time and
\bqn\label{est}
\|u(t)-U_\lambda\|_{W_{q,D}^2(I)}+\|\partial_t u(t)\|_{L_{q}(I)} \le R e^{-\omega_0 t} \|u^0-U_\lambda\|_{W_{q,D}^2(I)}\ ,\quad t\ge 0\ .
\eqn
\end{itemize}
\end{thm}

The first part of Theorem~\ref{TStable} is a consequence of the Implicit Function Theorem while the second part follows from the Principle of Linearized Stability, and the proofs are given in Section~\ref{SectTStable}. 
We shall point out that Theorem~\ref{TStable} provides uniqueness of steady-states with first components {\it in} $S_q(\kappa)$ for fixed $\lambda$ small. A result in this spirit is also shown in \cite[Thm.5.6]{GhoussoubGuoSIMA07}.} But, as pointed out before, for the small aspect ratio model \eqref{z0} it is known that below the critical threshold there are exactly two steady-states. If this would turn out to be true for the free boundary problem as well, that is, if there would be another smooth branch of steady-states emanating from $\lambda=0$, say, $V_\lambda\not= U_\lambda$, then the fact that $S_q(\kappa_1)\subset S_q(\kappa_2)$ for $0<\kappa_2<\kappa_1<1$ would imply that $\delta(\kappa)\searrow 0$ as $\kappa\searrow 0$ in Theorem~\ref{TStable}. Obviously, $V_\lambda\notin S_q(\kappa)$ for $\lambda<\delta(\kappa)$ and thus, as $\lambda\searrow 0$, the minimum of $V_\lambda$ has to approach $-1$ or the $W_q^2$-norm of $V_\lambda$ has to blow up\footnote{For the case of the small aspect ratio model, $V_\lambda$ approaches the $V$-shaped function $x\mapsto\vert x\vert-1$ as $\lambda\searrow 0$.}.

We also note that $\psi$ converges exponentially to $\Psi_\lambda$ in the $W_2^2$-norm as $t\rightarrow \infty$, see Corollary~\ref{RR1} for a precise statement. Finally, both components of the steady-state enjoy more regularity than stated, see \cite[Cor.10]{LaurencotWalker_ARMA}.

\medskip

More insight in the connection between the free boundary model and its small aspect ratio limit is offered in the next theorem. Indeed, we show that the solution $(u,\psi)=(u_\ve,\psi_\ve)$ to \eqref{u}-\eqref{bcpsi} provided by Theorem~\ref{A} converges to the solution $(u_0,\psi_0)$ of the small aspect ratio model \eqref{psi0}, \eqref{z0} as $\varepsilon\rightarrow 0$. This gives a rigorous justification of the formal derivation.

\begin{thm}[{\bf Small Aspect Ratio Limit}]\label{Bq}
Let $\lambda>0$, $q\in (2,\infty)$, $\kappa\in (0,1)$, and let $u^0\in S_q(\kappa)$  with $u^0(x)\le 0$ for $x\in I$. For $\varepsilon>0$ let $(u_\varepsilon,\psi_\varepsilon)$ be the unique solution to \eqref{u}-\eqref{bcpsi} on the maximal interval of existence $[0,T_m^\varepsilon)$. There are $\tau>0$, $\ve_0>0$, and $\kappa_0\in (0,1)$ depending only on $q$ and $\kappa$ such that $\ T_m^\varepsilon\ge \tau$ and $u_\ve(t)\in S_q(\kappa_0)$ for all $(t,\ve)\in [0,\tau]\times (0,\ve_0)$. Moreover, the small aspect ratio equation \eqref{z0}
has a unique solution
$$
u_0\in C^1\big([0,\tau],L_q(I)\big)\cap C\big([0,\tau],W_{q,D}^2(I)\big)
$$
satisfying $u_0(t)\in S_q(\kappa_0)$ for all $t\in [0,\tau]$ and such that the convergences
$$ 
u_{\varepsilon}\longrightarrow u_0\quad \text{in}\quad C^{1-\theta}\big([0,\tau], W_q^{2\theta}(I)\big)\ ,\quad 0<\theta<1\ ,
$$
and
\begin{equation}\label{z00} 
\psi_{\varepsilon}(t)\mathbf{1}_{\Omega(u_{\varepsilon}(t))}\longrightarrow \psi_{0}(t)\mathbf{1}_{\Omega(u_{0}(t))}\quad \text{in}\quad L_2\big(I\times (-1,0)\big)\ ,\quad t\in [0,\tau]\ ,
\end{equation}
hold as $\ve\rightarrow 0$, where $\psi_0$ is the potential given in \eqref{psi0}. Furthermore, there is $\Lambda(\kappa)>0$ such that the results above hold true for each  $\tau>0$ provided that $\lambda\in (0,\Lambda(\kappa))$.
\end{thm}

A similar result has been established for the stationary problem in \cite[Theorem~2]{LaurencotWalker_ARMA} and the proof of Theorem~\ref{Bq} is performed along the same lines provided one ensures an $\ve$-independent lower bound $\tau>0$ on $T_m^\ve$. In addition, in \cite{LaurencotWalker_ARMA} we took advantage of the fact that a $W_\infty^2(I)$-bound is available for solutions to the stationary problem. We refine the arguments here by showing that a $W_q^2(I)$-bound is sufficient for $q>2$.

\section{Local and Global Well-Posedness: Proof of Theorem~\ref{A} and Theorem~\ref{TD}($\mathrm{i}$) }\label{Sec2} 

The starting point for the proof of Theorem~\ref{A} is to transform the free boundary problem \eqref{psi}-\eqref{bcpsi} to the fixed rectangle $\Omega:=I\times (0,1)$. More precisely, let $q>2$ be fixed and consider an arbitrary function $v\in W_{q,D}^2(I)$ taking values in $(-1,\infty)$. We then define a diffeomorphism \mbox{$T_v:=\overline{\Omega(v)}\rightarrow \overline{\Omega}$} by setting
\begin{equation}\label{Tu}
T_v(x,z):=\left(x,\frac{1+z}{1+v(x)}\right)\ ,\quad (x,z)\in \overline{\Omega(v)}
\end{equation}
with $\Omega(v) = \left\{ (x,z)\in I\times (-1,\infty)\ ;\quad -1 < z < v(x) \right\}$. Clearly, its inverse is
\begin{equation}\label{Tuu}
T_v^{-1}(x,\eta)=\big(x,(1+v(x))\eta-1\big)\ ,\quad (x,\eta)\in \overline{\Omega}\ ,
\end{equation}
and the Laplace operator is transformed to the $v$-dependent differential operator 
\begin{equation*}
\begin{split}
\mathcal{L}_v w\, :=\, & \e^2\ \partial_x^2 w - 2\e^2\ \eta\ \frac{\partial_x v(x)}{1+v(x)}\ \partial_x\partial_\eta w
+ \frac{1+\e^2\eta^2(\partial_x v(x))^2}{(1+v(x))^2}\ \partial_\eta^2 w\\
& + \e^2\ \eta\ \left[ 2\ \left(\frac{\partial_x v(x)}{1+v(x)} \right)^2 - \frac{\partial_x^2 v(x)}{1+v(x)} \right]\ \partial_\eta w\ .
\end{split}
\end{equation*}
The boundary value problem  \eqref{psi}-\eqref{bcpsi} is then obviously equivalent to
\begin{eqnarray}
\big(\mathcal{L}_{u(t)}\phi\big) (t,x,\eta)\!\!\!&=0\ ,&(x,\eta)\in\Omega\ , \quad t>0\ ,\label{23}\\
\phi(t,x,\eta)\!\!\!&=\eta\ , &(x,\eta)\in \partial\Omega\ , \quad t>0\ ,\label{24}
\end{eqnarray}
for $\phi=\psi\circ T_{u(t)}^{-1}$.
With this notation, the evolution equation \eqref{u} for $u$ becomes
\begin{align}
\partial_t u -\partial_x^2 u &= {-\lambda}\ \left[ \frac{1+\e^2(\partial_x u)^2}{(1+u)^2} \right]\ \vert\partial_\eta\phi(\cdot,1)\vert^2\ ,\quad x\in I\ ,\quad t>0\ , \label{33}
\end{align}
after noticing that we have $\partial_x\phi(t,x,1)=0$ for $x\in I$ and $t>0$ due to $\phi(t,x,1)=1$ by \eqref{24}. To set the stage for the proof of Theorem~\ref{A} we first observe:

\begin{prop}\label{L1}
Let $\kappa\in (0,1)$ and $\varepsilon>0$. For each $v\in S_q(\kappa)$ there is a unique solution $\phi_v\in W_2^2(\Omega)$ to 
\begin{eqnarray}
\big(\mathcal{L}_v \phi_v\big) (x,\eta)\!\!\!&=0\ ,&(x,\eta)\in\Omega\ ,\label{230}\\
\phi_v(x,\eta)\!\!\!&=\eta\ , &(x,\eta)\in \partial\Omega\ .\label{240}
\end{eqnarray}
In addition, defining $\tilde{v}$ by $\tilde{v}(x):=v(-x)$ for $x\in I$, we have $\phi_{\tilde{v}}(x,\eta)=\phi_v(-x,\eta)$ for $(x,\eta)\in\Omega$. Moreover, for $2\sigma\in [0,1/2)$, the mapping
$$
g_\ve: S_q(\kappa)\longrightarrow W_{2,D}^{2\sigma}(I)\ ,\quad v\longmapsto \frac{1+\e^2(\partial_x v)^2}{(1+v)^2}\  \vert\partial_\eta\phi_v(\cdot,1)\vert^2
$$
is {analytic,} globally Lipschitz continuous, and bounded with $g_\ve(0)=1$.
\end{prop}

The proof of Proposition~\ref{L1} shares some common steps with that of \cite[Lem.~5 \&~6]{LaurencotWalker_ARMA}, but requires further developments, in particular establishing the Lipschitz continuity of $g_\ve$ which was not needed in \cite{LaurencotWalker_ARMA}. We first derive suitable properties of the operator $\mathcal{L}_v$ for $v$ in the closure $$\overline{S}_q(\kappa)=\left\{u\in W_{q,D}^2(I)\,;\, \|u\|_{W_{q,D}^2(I)}\le 1/\kappa \;\;\text{ and }\;\; -1+\kappa\le u(x) \text{ for } x\in I \right\}
$$ 
of $S_q(\kappa)$, which we gather in the next lemma.

\begin{lem}\label{L2}
Let $\kappa\in (0,1)$ and $\varepsilon>0$. For each $v\in \overline{S}_q(\kappa)$ and $F\in L_2(\Omega)$, there is a unique solution $\Phi\in W_{2,D}^2(\Omega)$ to the boundary value problem 
\begin{eqnarray}
-\mathcal{L}_v\Phi  \!\!\!&=F\ &\text{in}\ \Omega\ ,\label{231}\\
\Phi \!\!\!&=0\  &\text{on}\ \partial\Omega\ .\label{241}\ 
\end{eqnarray}
Moreover, there is a constant $c_1(\kappa,\varepsilon)>0$ depending only on $q$, $\kappa$, and $\varepsilon$ such that
\begin{equation}
\|\Phi\|_{W_2^2(\Omega)} \le c_1(\kappa,\varepsilon)\ \|F\|_{L_2(\Omega)}\, . \label{y3}
\end{equation}
\end{lem}

\begin{proof}[{\bf Proof}]
First note that the definition of $\overline{S}_q(\kappa)$ and Sobolev's embedding theorem guarantee the existence of some constant $c_0>0$ depending only on $q$ such that, for $v\in \overline{S}_q(\kappa)$,
\begin{equation}\label{c2}
1+v(x)\ge \kappa\ ,\quad x\in I\ ,\quad\text{ and }\;\; \|v\|_{C^1([-1,1])}\le \frac{c_0}{\kappa}\ .
\end{equation}
It follows from the proof of \cite[Lem.~5]{LaurencotWalker_ARMA} that, due to \eqref{c2}, the operator $-\mathcal{L}_v$ is elliptic with ellipticity constant $\nu(\kappa,\varepsilon)>0$ being independent of $v\in  \overline{S}_q(\kappa)$. Moreover, writing $-\mathcal{L}_v$ in divergence form, 
\begin{align*}
-\mathcal{L}_v w =& -\partial_x \left( a_{11}(v)\ \partial_x w + a_{12}(v)\ \partial_\eta w \right) -\partial_\eta \left( a_{21}(v)\ \partial_x w + a_{22}(v)\ \partial_\eta w \right) \\
& + b_1(v)\  \partial_x w + b_2(v)\  \partial_\eta w\ ,
\end{align*}
with
\begin{align*}
a_{11}(v) := \varepsilon^2\,, \qquad & \qquad a_{22}(v) := \frac{1+\varepsilon^2\ \eta^2\ |\partial_x v(x)|^2}{(1+v(x))^2} \,, \\
a_{12}(v) := - \varepsilon^2\ \eta\ \frac{\partial_x v(x)}{1+v(x)}\,, \qquad & \qquad a_{21}(v) := a_{12}(v)\,, \\
b_1(v) := \varepsilon^2\ \frac{\partial_x v(x)}{1+v(x)}\,, \qquad & \qquad b_2(v) := - \varepsilon^2\ \eta\ \left( \frac{\partial_x v(x)}{1+v(x)} \right)^2\ ,
\end{align*}
we see from \eqref{c2} and the definition of $\overline{S}_q(\kappa)$ that 
\begin{equation}
\sum_{i,j=1}^2 \|a_{ij}(v)\|_{W_q^1(\Omega)} + \sum_{i=1}^2 \|b_i(v)\|_{L_\infty(\Omega)} \le c_2(\kappa,\varepsilon) \label{y4b}
\end{equation} 
for all $v\in \overline{S}_q(\kappa)$. Moreover, the embedding of $W_q^1(I)$ in $C([-1,1])$ ensures that $a_{ij}(v)$ belongs to $C(\overline{\Omega})$ for $1\le i,j\le 2$ and $v\in \overline{S}_q(\kappa)$. It then follows from \cite[Thm.~8.3]{GilbargTrudinger} that, given $v\in \overline{S}_q(\kappa)$ and $F\in L_2(\Omega)$, the boundary value problem \eqref{231}--\eqref{241} has a unique weak solution $\Phi\in W_{2,D}^1(\Omega)$. Furthermore, the regularity of $\Phi$ and \eqref{y4b} ensure that $G:= F - b_1(v)\ \partial_x\Phi - b_2(v)\ \partial_\eta\Phi$ belongs to $L_2(\Omega)$, and we are in a position to apply \cite[Chapt.~3, Thm.~9.1]{LadyzhenskayaUraltsevaLQEE} to conclude that $\Phi$ is actually the unique solution in $W_{2,D}^2(\Omega)$ to the boundary value problem
$$
-\mathcal{L}_v^0 \Phi  =G\ \text{ in }\ \Omega\ ,\qquad
\Phi =0\  \text{ on }\ \partial\Omega\ ,
$$
where $\mathcal{L}_v^0$ denotes the principal part of the operator $\mathcal{L}_v$, that is,
$$
-\mathcal{L}_v^0 w\, := -\partial_x \left( a_{11}(v)\ \partial_x w + a_{12}(v)\ \partial_\eta w \right) -\partial_\eta \left( a_{21}(v)\ \partial_x w + a_{22}(v)\ \partial_\eta w \right)\ .
$$
In addition, it follows from \cite[Chapt.~3, Thm.~10.1]{LadyzhenskayaUraltsevaLQEE} that there is a constant $c_3(\kappa,\varepsilon)>0$ depending only on $q$, $\kappa$, and $\varepsilon$ such that
$$
\|\Phi\|_{W_2^2(\Omega)} \le c_3(\kappa,\varepsilon)\ \left( \|\Phi\|_{L_2(\Omega)} + \|G\|_{L_2(\Omega)} \right)\ .
$$
Combining the previous inequality with \eqref{y4b} and the inequality
\begin{align*}
\|\partial_x\Phi\|_{L_2(\Omega)}^2 + \|\partial_\eta\Phi\|_{L_2(\Omega)}^2  =  - \int_\Omega \Phi\ \left( \partial_x^2\Phi + \partial_\eta^2\Phi \right) \rd (x,\eta)  
\le  \delta^2\ \|\Phi\|_{W_2^2(\Omega)}^2 + \frac{1}{4\delta^2}\ \|\Phi\|_{L_2(\Omega)}^2
\end{align*} 
which is valid for all $\delta>0$, we are led to
\begin{align*}
\|\Phi\|_{W_2^2(\Omega)} & \le  c_3(\kappa,\varepsilon)\ \left[ \|\Phi\|_{L_2(\Omega)} + \|F\|_{L_2(\Omega)} + c_2(\kappa,\varepsilon)\ \left( \|\partial_x\Phi\|_{L_2(\Omega)} + \|\partial_\eta\Phi\|_{L_2(\Omega)} \right) \right] \\
& \le  c_3(\kappa,\varepsilon)\ \left[ \|\Phi\|_{L_2(\Omega)} + \|F\|_{L_2(\Omega)} + 2 c_2(\kappa,\varepsilon)\ \left( \delta\ \|\Phi\|_{W_2^2(\Omega)} + \frac{1}{2\delta}\ \|\Phi\|_{L_2(\Omega)} \right) \right] \,
\end{align*} 
whence, after choosing $\delta$ sufficiently small,
\begin{equation}
\|\Phi\|_{W_2^2(\Omega)} \le c_4(\kappa,\varepsilon)\ \left( \|\Phi\|_{L_2(\Omega)} + \|F\|_{L_2(\Omega)} \right)\ . \label{y5}
\end{equation}
We finally prove \eqref{y3} and argue as in the proof of \cite[Lemma~9.17]{GilbargTrudinger}. Assume for contradiction that \eqref{y3} is not true. Then, for each $n\ge 1$, there are $v_n\in \overline{S}_q(\kappa)$, $\widetilde{\Phi}_n\in W_{2,D}^2(\Omega)$, $\widetilde{\Phi}_n\not\equiv 0$, and $\widetilde{F}_n\in L_2(\Omega)$ such that
\begin{equation}
- \mathcal{L}_{v_n} \widetilde{\Phi}_n  =\widetilde{F}_n\ \text{  in }\ \Omega \;\;\;\text{ and }\;\;\; \left\| \widetilde{\Phi}_n \right\|_{W_2^2(\Omega)} \ge n\ \left\| \widetilde{F}_n \right\|_{L_2(\Omega)}\ . \label{y10}
\end{equation}
Setting $\Phi_n := \widetilde{\Phi}_n/ \left\| \widetilde{\Phi}_n \right\|_{L_2(\Omega)}$ and $F_n := \widetilde{F}_n/ \left\| \widetilde{\Phi}_n \right\|_{L_2(\Omega)}$, we realize that \eqref{y5} and \eqref{y10} imply 
\begin{eqnarray}
- \mathcal{L}_{v_n} \Phi_n \!\! & = &\!\! F_n\ \text{  in }\ \Omega\ , \qquad \Phi_n\in W_{2,D}^2(\Omega)\ , \label{y11a}\\
\left\| \Phi_n \right\|_{L_2(\Omega)} \!\! & = &\!\! 1\ , \label{y11b} 
\end{eqnarray}
and
\begin{align*}
n \|F_n\|_{L_2(\Omega)} & \le  \|\Phi_n\|_{W_2^2(\Omega)} \le c_4(\kappa,\varepsilon)\ \left( \|\Phi_n\|_{L_2(\Omega)} + \|F_n\|_{L_2(\Omega)} \right)  {=}\,  c_4(\kappa,\varepsilon)\ \left( 1 + \|F_n\|_{L_2(\Omega)} \right)\ .
\end{align*}
Consequently, we have for $n\ge 2c_4(\kappa,\varepsilon)$,
\begin{equation}
n \|F_n\|_{L_2(\Omega)} \le 2 c_4(\kappa,\varepsilon) \;\;\;\text{ and }\;\;\; \|\Phi_n\|_{W_2^2(\Omega)} \le \left( 1 + \frac{2}{n} \right)\ c_4(\kappa,\varepsilon)\ . \label{y12}
\end{equation}
Since $W_2^2(\Omega)$ and $W_q^2(I)$ are compactly embedded in $W_2^1(\Omega)$ and $C^1([-1,1])$, respectively, we infer from \eqref{y12} and the boundedness of $\overline{S}_q(\kappa)$ in $W_q^2(I)$ that there are $(\Phi,v)\in W_{2,D}^2(\Omega)\times W_{q,D}^2(I)$ and a subsequence of $(\Phi_n,v_n)_n$ (not relabeled) such that 
\begin{eqnarray}
& & (\Phi_n,v_n) \rightharpoonup (\Phi,v) \;\;\text{ in }\;\; W_{2,D}^2(\Omega)\times W_{q,D}^2(I)\ , \label{y13} \\
& & (\Phi_n,v_n) \longrightarrow (\Phi,v) \;\;\text{ in }\;\; W_{2,D}^1(\Omega)\times C^1([-1,1])\ . \label{y14}
\end{eqnarray}
It readily follows from \eqref{y13} and \eqref{y14} that $v\in \overline{S}_q(\kappa)$ and
\begin{equation}
\left( a_{ij}(v_n) , b_i(v_n) \right) \longrightarrow \left( a_{ij}(v), b_i(v) \right) \;\;\text{ in }\;\; C([-1,1]) \label{y15}
\end{equation}
for all $1\le i,j\le 2$. Since $F_n\longrightarrow 0$ in $L_2(\Omega)$ by \eqref{y12}, the convergences \eqref{y13}, \eqref{y14}, and \eqref{y15} allow us to pass to the limit as $n\to\infty$ in the weak formulation of \eqref{y11a} and conclude that $\Phi\in W_{2,D}^2(\Omega)$ is a weak solution to $-\mathcal{L}_v\Phi = 0$ in $\Omega$. Using again \cite[Thm.~8.3]{GilbargTrudinger}, this implies $\Phi\equiv 0$ contradicting $\|\Phi\|_{L_2(\Omega)}=1$ as follows from \eqref{y11b} and \eqref{y14}.
\end{proof}

\begin{proof}[{\bf Proof of Proposition~\ref{L1}}]
For $v\in S_q(\kappa)$ and $(x,\eta)\in\Omega$, we set 
$$
f_v(x,\eta):=\mathcal{L}_v \eta = \varepsilon^2\ \eta \left[ 2\ \left( \frac{\partial_x v(x)}{1+v(x)} \right)^2 - \frac{\partial_x^2 v(x)}{1+v(x)} \right]\ .
$$
Since $q>2$, the function $f_v$ belongs to $L_2(\Omega)$ with 
\begin{equation}
\| f_v \|_{L_2(\Omega)} \le c_5(\kappa,\varepsilon)\ , \label{y19}
\end{equation}
and Lemma~\ref{L2} ensures that there is a unique solution $\Phi_v\in W_{2,D}^2(\Omega)$ to
\begin{eqnarray}
-\mathcal{L}_v\Phi_v  \!\!\!&=f_v\ &\text{in}\ \Omega\ ,\label{23a}\\
\Phi_v \!\!\!&=0\  &\text{on}\ \partial\Omega\ ,\label{24a}\ 
\end{eqnarray}
satisfying
\begin{equation}
\|\Phi_v\|_{W_2^2(\Omega)}\le c_1(\kappa,\varepsilon)\ {  \|f_v\|_{L_2(\Omega)}}\ . \label{gaston}
\end{equation}
Letting $\phi_v(x,\eta)=\Phi_v(x,\eta)+\eta$ for $(x,\eta) \in \overline{\Omega}$, the function $\phi_v$ obviously solves \eqref{230}-\eqref{240}, { and from \eqref{y19} and \eqref{gaston} we obtain}
\begin{equation}\label{c4}
\|\phi_v\|_{W_2^2(\Omega)}\le c_6(\kappa,\varepsilon)\ .
\end{equation}
In addition, if $v\in S_q(\kappa)$ and $\tilde{v}$ denotes the function defined by $\tilde{v}(x) := v(-x)$ for $x\in I$, we obviously have $\tilde{v}\in S_q(\kappa)$ and the properties of $\mathcal{L}_{\tilde{v}}$ ensure that $(x,\eta)\longmapsto \phi_v(-x,\eta)$ solves \eqref{23a}-\eqref{24a}  with $\tilde{v}$ instead of $v$. The uniqueness of the solution to \eqref{23a}-\eqref{24a} then readily implies that $\phi_{\tilde{v}}(x,\eta)= \phi_v(-x,\eta)$ for $(x,\eta)\in\Omega$ and thus that \begin{equation}
g_\ve(\tilde{v})(x) = g_\ve(v)(-x)\,, \qquad x\in I\ . \label{symmetry}
\end{equation}
Next, given $v\in S_q(\kappa)$, we define a bounded linear operator $\mathcal{A}(v)\in\mathcal{L}\big(W_{2,D}^2(\Omega), L_2(\Omega)\big)$ by
$$
\mathcal{A}(v)\Phi:=-\mathcal{L}_v\Phi\ ,\quad \Phi\in W_{2,D}^2(\Omega)\ .
$$
Lemma~\ref{L2} guarantees that $\mathcal{A}(v)$ is invertible with inverse $\mathcal{A}(v)^{-1} \in \mathcal{L}\big(L_2(\Omega),W_{2,D}^2(\Omega))$ satisfying
\begin{equation}
\left\| \mathcal{A}(v)^{-1} \right\|_{\mathcal{L}\big(L_2(\Omega),W_{2,D}^2(\Omega))} \le c_1(\kappa,\varepsilon)\ . \label{y25}
\end{equation}
We then note that
\begin{equation}
\|\mathcal{A}(v_1)-\mathcal{A}(v_2)\|_{\mathcal{L}(W_{2,D}^2(\Omega),L_2(\Omega))}\le c_7(\kappa,\varepsilon)\ \|v_1-v_2\|_{W_q^2(I)}\ ,\quad v_1, v_2\in S_q(\kappa)\ ,\label{y26}
\end{equation}
which follows from the definition of $\mathcal{L}_v $ and the continuity of pointwise multiplication
$$
W_q^1(I)\cdot W_q^1(I)\hookrightarrow W_q^1(I) \hookrightarrow L_\infty(I)
$$
except for the terms involving $\partial_x^2 v_i$, $i=1,2$, where continuity of pointwise multiplication
$$
L_q(\Omega)\cdot W_2^1(\Omega)\hookrightarrow L_2(\Omega)
$$
is used. Now, for $v_1, v_2\in S_q(\kappa)$, we infer from \eqref{y25} and \eqref{y26} that
\begin{align*}
 \|\mathcal{A}(v_1)^{-1}&   -\mathcal{A}(v_2)^{-1}\|_{\mathcal{L}(L_2(\Omega),W_{2,D}^2(\Omega))} \\ 
\le & \left\| \mathcal{A}(v_1)^{-1} \right\|_{\mathcal{L}(L_2(\Omega),W_{2,D}^2(\Omega))}\ \left\| \mathcal{A}(v_2) - \mathcal{A}(v_1) \right\|_{\mathcal{L}(W_{2,D}^2(\Omega),L_2(\Omega))}\ \left\| \mathcal{A}(v_2)^{-1} \right\|_{\mathcal{L}(L_2(\Omega),W_{2,D}^2(\Omega))} \\ 
\le & c_1(\kappa,\varepsilon)^2\ c_7(\kappa,\varepsilon)\ \|v_1-v_2\|_{W_q^2({I})}\ ,
\end{align*}
which, combined with \eqref{y19}, the observation that $0\in S_q(\kappa)$ and
$$
\|f_{v_1}-f_{v_2}\|_{L_2(\Omega)}\le  c_8(\kappa,\varepsilon)\ \|v_1-v_2\|_{W_q^2({I})}\ ,\quad v_1, v_2\in S_q(\kappa)\ ,
$$
ensures that 
\begin{align}
\|\phi_{v_1}-\phi_{v_2}\|_{W_{2}^2(\Omega)} =\ & \|\Phi_{v_1}-\Phi_{v_2}\|_{W_{2}^2(\Omega)} = \|\mathcal{A}(v_1)^{-1} f_{v_1} - \mathcal{A}(v_2)^{-1} f_{v_2}\|_{W_{2,D}^2(\Omega)} \nonumber \\
\le\ & c_1(\kappa,\varepsilon)^2\ c_7(\kappa,\varepsilon)\ \|v_1-v_2\|_{W_q^2(I)}\ \|f_{v_1} \|_{L_2(\Omega)} \nonumber\\
& +\ c_8(\kappa,\varepsilon)\ \|v_1-v_2\|_{W_q^2(I)}\ \|\mathcal{A}(v_2)^{-1}\|_{\mathcal{L}(L_2(\Omega),W_{2,D}^2(\Omega))}\nonumber\\
\le\ & c_9(\kappa,\varepsilon)\ \|v_1-v_2\|_{W_q^2(I)}\label{RR}
\end{align}
for $v_1, v_2 \in S_q(\kappa)$. We may then invoke \cite[Chapt.~2, Thm.~5.4]{Necas67} and the continuity of pointwise multiplication
$$
W_2^{1/2}(I)\cdot W_2^{1/2}(I)\hookrightarrow W_2^{2\sigma_1}(I)
$$
for $2\sigma_1< 1/2$ according to \cite[Thm.~4.1]{AmannMultiplication} to conclude that the mapping
\begin{equation}
S_q(\kappa)\rightarrow W_2^{2\sigma_1}(I)\ ,\quad v\mapsto \big\vert\partial_\eta \phi_v(\cdot,1)\big\vert^2 \label{c12a}
\end{equation}
is globally Lipschitz continuous. { Thanks} to the continuity of the embedding of $W_q^2(I)$ in $W_\infty^1(I)$, the mapping
\begin{equation}
S_q(\kappa)\rightarrow W_q^1(I)\ ,\quad v\mapsto \frac{1+\varepsilon^2(\partial_x v)^2}{(1+v)^2} \label{c12b}
\end{equation}
is globally Lipschitz continuous with a Lipschitz constant depending only on $\kappa$ and $\varepsilon$, and the Lipschitz continuity of $g_\ve$ stated in Proposition~\ref{L1} follows at once from \eqref{c12a}, \eqref{c12b}, and continuity of pointwise multiplication
$$
W_q^{1}(I)\cdot W_2^{2\sigma_1}(I)\hookrightarrow W_2^{2\sigma}(I)=W_{2,D}^{2\sigma}(I)\ ,
$$
where $2\sigma<2\sigma_1<1/2$, see again \cite[Thm.~4.1]{AmannMultiplication}. {Finally, to prove that $g_\ve$ is analytic, we note that $S_q(\kappa)$ is open in $W_{q,D}^2(I)$ and that the mappings $\mathcal{A}: S_q(\kappa)\rightarrow\mathcal{L}(W_{2,D}^2(\Omega),L_2(\Omega))$ and $[v\mapsto f_v]: S_q(\kappa)\rightarrow L_2(\Omega)$ are analytic. The analyticity of the inversion map $\ell\mapsto\ell^{-1}$ for bounded operators implies that also the mapping $[v\mapsto \phi_v]:S_q(\kappa)\rightarrow W_2^2(\Omega)$ is analytic, and the assertion follows as above from the results on pointwise multiplication .}
\end{proof}

\medskip

Let $p\in (1,\infty)$. We define $A_p\in \mathcal{L}(W_{p,D}^2(I),L_p(I))$ by $A_pv:=-\partial_x^2 v$ for \mbox{$v\in W_{p,D}^2(I)$}. Since $A_r\subset A_p$ for $r\ge p$, we suppress the subscript in the following and write $A:=A_p$.
Note then that \eqref{33} subject to the boundary condition \eqref{bcu} and the initial condition \eqref{ic} may be recast as an abstract parameter-dependent Cauchy problem
\bqn\label{CP}
\dot{u}+Au=-\lambda g_\ve(u)\ ,\quad t>0\ ,\qquad u(0)=u^0\ ,
\eqn
where we recall that the function $g_\ve$ was defined {in} Proposition~\ref{L1}. To prove Theorem~\ref{A} it then suffices to focus on \eqref{CP}. {For that purpose}, let $\{e^{-tA}\,;\, t\ge 0\}$ denote the heat semigroup on $L_p(I)$ corresponding to $-A$.
In order to state suitable regularizing properties we recall that we have set
$ W_{p,D}^{2\beta}(I)=W_{p}^{2\beta}(I)$ for $2\beta\in (0,1/p)$ and $ W_{p,D}^{2\beta}(I)=\{u\in  W_{p}^{2\beta}(I)\,;\, u(\pm 1)=0\}$ for $2\beta\in (1/p,2]$. Then we have:

\begin{lem}\label{regprop}
Let $1<p\le r<\infty$. There exists $\omega>0$ such that the following hold.
\begin{itemize}
\item[(i)] If $0\le \alpha\le\beta\le 1$ with $2\alpha, 2\beta\not= 1/p$, then
$$
\|e^{-tA}\|_{\mathcal{L}(W_{p,D}^{2\alpha}(I),W_{p,D}^{2\beta}(I))}\le M e^{-\omega t} t^{\alpha-\beta}\ ,\quad t>0\ ,
$$
for some number $M\ge 1$ depending on $p$, $\alpha$, and $\beta$.\\

\item[(ii)] If $0\le \alpha\le \beta\le 1$ with $2\alpha\not=1/p$ and $2\beta\not= 1/r$, then
$$
\|e^{-tA}\|_{\mathcal{L}(W_{p,D}^{2\alpha}(I),W_{r,D}^{2\beta}(I))}\le M e^{-\omega t} t^{\alpha-\beta-\frac{1}{2}(\frac{1}{p}-\frac{1}{r})}\ ,\quad t>0\ ,
$$
for some number $M\ge 1$ depending on $p$, $r$, $\alpha$, and $\beta$.
\end{itemize}
\end{lem}

\begin{proof}[{\bf Proof}]
(i) Since $-A\in \mathcal{L}(W_{p,D}^2(I),L_p(I))$ is the generator of the analytic semigroup $\{e^{-tA}\,;\, t\ge 0\}$ on $L_p(I)$ with a negative spectral bound, it follows from \cite[{Chapt.~V,} Thm.2.1.3]{LQPP} that there are $\omega>0$ and $M\ge 1$ such that
$$
\|e^{-tA}\|_{\mathcal{L}(E_\alpha,E_\beta)}\le M e^{-\omega t} t^{\alpha-\beta}\ ,\quad t>0\ ,
$$
where, for $\theta\in [0,1]$, $E_\theta:=(L_p(I),W_{p,D}^{2}(I))_{\theta}$ with $(\cdot,\cdot)_{\theta}$ chosen as real interpolation functor $(\cdot,\cdot)_{\theta,p}$ if $2\theta\not= 1$ and as complex interpolation functor $[\cdot,\cdot]_{1/2}$ if $2\theta=1$. Since $E_\theta = W_{p,D}^{2\theta}(I)$ with equivalent norms for $2\theta\in [0,2]\setminus\{1/p\}$ by \cite{Grisvard69, Seeley}, assertion~(i) follows.\\

(ii) From Sobolev's embedding we have $W_{p,D}^{2}(I)\hookrightarrow W_{r,D}^{2\theta}(I)$ for $2\theta=2-(1/p-1/r)\not= 1/r$, whence
$$
\|e^{-tA}\|_{\mathcal{L}(W_{p,D}^{2\alpha}(I),W_{r,D}^{2\beta}(I))} \,\le\, c\, \|e^{-\frac{t}{2}A}\|_{\mathcal{L}(W_{r,D}^{2\theta}(I),W_{r,D}^{2\beta}(I))}\, \|e^{-\frac{t}{2}A}\|_{\mathcal{L}(W_{p,D}^{2\alpha}(I),W_{p,D}^{2}(I))}\ ,\quad t>0\ ,
$$
and so assertion (ii) follows from (i).
\end{proof}

We are now in a position to prove Theorem~\ref{A} and Theorem~\ref{TD}(i).

\begin{proof}[{\bf Proof of Theorem~\ref{A} and Theorem~\ref{TD}(i)}]
Let $\lambda> 0$, $q\in (2,\infty)$, $\varepsilon>0$, and consider $u^0\in  W_{q,D}^2(I)$  with $u^0(x)>-1$ for $x\in I$. Clearly, there is $\kappa\in (0,1/2)$ such that 
\bqn\label{44}
u^0\in { \overline{S}_q(2\kappa)}\ .
\eqn
We now fix $\frac{1}{2}-\frac{1}{q}<2\sigma<\frac{1}{2}$  with $2\sigma\ne 1/q$ and put $\kappa_0:=\kappa/M$, where $M\ge 1$ is such that
\bqn\label{9}
\|e^{-tA}\|_{\mathcal{L}(W_{q,D}^{2}(I))}+ t^{-\sigma+1+\frac{1}{2}(\frac{1}{2}-\frac{1}{q})} \|e^{-tA}\|_{\mathcal{L}(W_{2,D}^{2\sigma}(I),W_{q,D}^{2}(I))}\le M e^{-\omega t} \ ,\quad t\ge 0\ 
\eqn
with $\omega>0$ according to Lemma~\ref{regprop}. By Proposition~\ref{L1} there is $c_{10}(\kappa,\varepsilon)>0$ such that
\bqn\label{10}
\|g_\ve(v_1)-g_\ve(v_2)\|_{W_{2,D}^{2\sigma}(I)}\le c_{10}(\kappa,\varepsilon)\ \|v_1-v_2\|_{W_{q,D}^{2}(I)}\ ,\quad v_1, v_2\in { \overline{S}_q}(\kappa_0)\ .
\eqn
Since $0\in { \overline{S}_q}(\kappa_0)$ and $g_\ve(0)=1$, we deduce from \eqref{10} that
\bqn\label{10a}
\|g_\ve(v)\|_{W_{2,D}^{2\sigma}(I)}\le 1 + \frac{c_{10}(\kappa,\varepsilon)}{\kappa_0} = c_{11}(\kappa,\varepsilon)\ ,\quad v\in { \overline{S}_q}(\kappa_0)\ .
\eqn
Now, for $\tau>0$, define $\mathcal{V}_\tau:= C([0,\tau], { \overline{S}_q}(\kappa_0))$ and
$$
F(v)(t):=e^{-tA}u^0-\lambda\int_0^t e^{-(t-s)A} g_\ve\big(v(s)\big)\,\rd s
$$
for $0\le t\le \tau$ and $v\in\mathcal{V}_\tau$. Consider $v_1, v_2\in \mathcal{V}_\tau$ and $t\in [0,\tau]$. Then, introducing
\begin{equation}
\mathcal{I}(\tau) := \int_0^\tau e^{-\omega s}\ s^{\sigma-1-\frac{1}{2}(\frac{1}{2}-\frac{1}{q})}\,\rd s \le  \mathcal{I}(\infty) := \int_0^\infty e^{-\omega s}\ s^{\sigma-1-\frac{1}{2}(\frac{1}{2}-\frac{1}{q})}\,\rd s \ , \label{volvic}
\end{equation}
which is finite thanks to the positivity of $\omega$ and the choice of $\sigma$, it readily follows from \eqref{44}, \eqref{9}, and \eqref{10a} that 
\begin{align}
\|F(v_1)(t)\|_{W_{q,D}^{2}(I)} & \le  M\ \|u^0\|_{W_{q,D}^2(I)} \nonumber 
+ \lambda\ M\ \int_0^t e^{-\omega(t-s)}\ (t-s)^{\sigma-1-\frac{1}{2}(\frac{1}{2}-\frac{1}{q})}\ \|g_\ve(v_1(s))\|_{W_{2,D}^{2\sigma}(I)}\,\rd s \nonumber \\
&\le  \frac{M}{2\kappa}+\lambda\ M\ c_{11}(\kappa,\varepsilon)\ \mathcal{I}(\tau)\ ,\label{11}
\end{align}
and from \eqref{9} and \eqref{10} that
\bqn\label{12}
\|F(v_1)(t)-F(v_2)(t)\|_{W_{q,D}^{2}(I)}\le \lambda\ M\ c_{10}(\kappa,\varepsilon)\ \mathcal{I}(\tau)\, \|v_1-v_2\|_{C([0,\tau], W_{q,D}^2(I))}\ .
\eqn
Moreover, since $W_{q,D}^{2}(I)$ embeds in $L_\infty(I)$ with embedding constant 2 and $u^0\ge 2\kappa-1$, we deduce from the positivity of the heat semigroup, \eqref{9}, and \eqref{10a} that
\begin{align}
F(v_1)(t)\ge & -1 + 2\kappa - 2\,\lambda\,  \int_0^t \left\| e^{-(t-s)A}\ g_\ve(v_1(s)) \right\|_{W_{q,D}^2(I)}\, \rd s \nonumber \\
\ge & -1+2\kappa -  2\,\lambda\, M \ \int_0^t e^{-\omega (t-s)}\ (t-s)^{\sigma-1-\frac{1}{2}(\frac{1}{2}-\frac{1}{q})}\ \|g_\ve(v_1(s))\|_{W_{2,D}^{2\sigma}(I)}\,\rd s \nonumber\\
\ge & -1+2\kappa -  2\,\lambda\,  M\ c_{11}(\kappa,\varepsilon)\ \mathcal{I}(\tau)\ .\label{13}
\end{align}
We finally note that $F(v_1)(t)\le 0$ if $u^0\le 0$ since $g_\ve(v_1)\ge 0$. Consequently, due to \eqref{11}-\eqref{13} and the fact that $\mathcal{I}(\tau)\to 0$ as $\tau\to 0$, there is $\tau_0:=\tau_0(\lambda,\kappa,\varepsilon,q,\sigma)>0$ sufficiently small such that $F$ defines a contraction from $\mathcal{V}_{\tau_0}$ into itself. This shows that there is a unique maximal solution 
$$
u\in C^1\big([0,T_m^\varepsilon),L_q(I)\big)\cap C\big([0,T_m^\varepsilon), W_{q,D}^2(I)\big)\cap C\big((0,T_m^\varepsilon),W_{2,D}^{2+2\sigma}(I)\big)
$$
to \eqref{CP} for some $T_m^\varepsilon\in (\tau_0,\infty]$, satisfying
$$
u(t,x)>-1\ ,\quad (t,x)\in [0,T_m^\varepsilon)\times I\ ,
$$
and, in addition, $$u(t,x)\le 0\ ,\quad(t,x)\in [0,T_m^\varepsilon)\times I\quad\text{if}\quad u^0(x)\le 0\ ,\quad x\in I\ .
$$
Moreover, if for each $\tau>0$ there is $\kappa(\tau)\in (0,1)$ such that $u(t)\in { \overline{S}_q}(\kappa(\tau))$ for $t\in [0,T_m^\varepsilon) \cap [0,\tau]$, then necessarily $T_m^\varepsilon=\infty$. This proves the statements (i) and (ii) of Theorem~\ref{A} after observing that $\psi(t):=\phi_{u(t)}\circ T_{u(t)}$ belongs to $W_q^2\big(\Omega(u(t))\big)$ and solves \eqref{psi}-\eqref{bcpsi} for each $t\in [0,T_m^\varepsilon)$, where the transformation $T_{u}$ was introduced in \eqref{Tu}.

As for the statement (i) of Theorem~\ref{TD}, we choose $\lambda_*:=\lambda_*(\kappa,\varepsilon,q,\sigma)>0$ such that (recall \eqref{volvic})
$$
\lambda_*\, M\, \max{\{ c_{10}(\kappa,\varepsilon) , c_{11}(\kappa,\varepsilon)\}}\, \mathcal{I}(\infty) \le \frac{1}{2} < \frac{1}{2\kappa_0} 
$$
and
$$  
2\,\lambda_*\, M\, c_{11}(\kappa,\varepsilon)\ \mathcal{I}(\infty) \le \kappa_0\ .
$$
Letting $\lambda\le \lambda_*$, it readily follows that, for each $\tau>0$, the mapping $F$ defines a contraction from $C([0,\tau], { \overline{S}_q}(\kappa_0))$ into itself. This implies that $T_m^\varepsilon=\infty$ in this case and that $u(t)\in { \overline{S}_q}(\kappa_0)$ for~$t\ge 0$.

To prove statement (iv) of Theorem~\ref{A} suppose that $u^0$ is even on $I$ and let $u$ be the corresponding maximal solution to \eqref{CP} with maximal existence time $T_m^\varepsilon\in (0,\infty]$. Introducing the function $\tilde{u}$ defined by $\tilde{u}(t,x)= u(t,-x)$ for $(t,x)\in [0,T_m^\varepsilon)\times I$, we deduce from Proposition~\ref{L1} and the evenness of $u^0$ that $\tilde{u}$ also solves \eqref{CP}, so that $\tilde{u}$ actually coincides with $u$. Thus $u(t,.)$ is even on $I$ for all $t\in [0,T_m^\varepsilon)$ and the proof of Theorem~\ref{A} is complete. 
\end{proof}

{We end this section with some useful properties of the component $\psi$ of solutions to \eqref{u}-\eqref{bcpsi}.

\begin{prop}\label{pr.psi}
Let $q\in (2,\infty)$, $\varepsilon>0$, $\lambda>0$, and consider an initial value $u^0\in  W_{q,D}^2(I)$ such that $u^0(x)>-1$ for $x\in I$. Denoting the corresponding maximal solution to \eqref{u}-\eqref{bcpsi} by $(u,\psi)$, we have for all $t\in [0,T_m^\varepsilon)$
\begin{align}
1+z - \max{\big( u^0 \big)_+} \le &\, \psi(t,x,z) \le 1\ , \qquad (x,z)\in \Omega(u(t))\ , \label{ppsi1} \\
\partial_z \psi(t,x,u(t,x)) \ge &\, 0\ , \qquad x\in I\ , \label{ppsi2} \\
\partial_x \psi(t,x,u(t,x))= & - \partial_x u(t,x)\ \partial_z \psi(t,x,u(t,x)) \ , \qquad x\in I\ . \label{ppsi3}
\end{align}
\end{prop}

\begin{proof}[{\bf Proof}]
It readily follows from the positivity of $\lambda$ that the constant $\max{\left( u_0 \right)_+}$ is a supersolution to \eqref{u}-\eqref{ic}, so that $u\le \max{\left( u_0 \right)_+}$ in $[0,T_m^\varepsilon)\times [-1,1]$. This property entails that $(x,z)\longmapsto 1+z- \max{\left( u^0 \right)_+}$ is a subsolution to \eqref{psi}-\eqref{bcpsi} and the comparison principle gives the lower bound in \eqref{ppsi1}. The upper bound in \eqref{ppsi1} also follows from the comparison principle as the constant $1$ is clearly a supersolution to \eqref{psi}-\eqref{bcpsi}. This implies in particular that $\psi(t)$ reaches its maximum value on the graph of $u(t)$ and thus that \eqref{ppsi2} holds true. Finally, \eqref{ppsi3} is an obvious consequence of \eqref{bcpsi}.
\end{proof}


{
\section{On Nonexistence of Global Solutions: Proof of Theorem~\ref{TD}($\mathrm{ii}$)}\label{Sec3a}

We now prove that there are no global solutions for large $\lambda$ values as stated in Theorem~\ref{TD}(ii) (note that part (i) of this theorem was shown in the previous section). For this we first need some preparations. Let $q\in (2,\infty)$, $\varepsilon>0$, $\lambda>0$, and consider an initial value $u^0\in  W_{q,D}^2(I)$ such that $-1<u^0(x)\le 0$ for $x\in I$. By Theorem~\ref{A}, there is a unique solution $(u,\psi)$ to \eqref{u}-\eqref{bcpsi} defined on the maximal interval of existence $[0,T_m^\varepsilon)$ for some $T_m^\varepsilon\in (0,\infty]$ and satysfying
$$
u\in C^1\big([0,T_m^\varepsilon),L_q(I)\big)\cap C\big([0,T_m^\varepsilon), W_{q,D}^2(I)\big)
$$
together with
\begin{equation}
-1<u(t,x)\le 0\ ,\quad (t,x)\in [0,T_m^\varepsilon)\times I\ , \label{n2}
\end{equation}
and $\psi(t)\in W_q^2\big(\Omega(u(t))\big)$ solves \eqref{psi}-\eqref{bcpsi} on $\Omega(u(t))$ for each $t\in [0,T_m^\varepsilon)$. Our aim is to show that, if $\lambda$ is sufficiently large, the maximal existence time $T_m^\varepsilon$ is finite. To this end, define $\zeta_1(x):= \pi \cos{(\pi x/2)}/4$ for $x\in [-1,1]$ and $\mu_1:=\pi^ 2/4$. Then, $\mu_1$ is the principal eigenvalue of the $L_2(I)$-realization of $-\partial_x^2$ and
\begin{equation}
-\partial_x^2 \zeta_1 = \mu_1\ \zeta_1 \;\;\text{ in }\;\; I\ , \quad \zeta_1(\pm 1)=0\ , \qquad \|\zeta_1\|_{L_1(I)}=1\ . \label{n2b}
\end{equation}
A classical technique to show that solutions only exist on a finite time interval is to study the evolution of 
$$
E_0(t) := \int_{-1}^1 \zeta_1(x)\ u(t,x)\ \rd x\ , \quad t\in [0,T_m^\varepsilon)\ ,
$$
and show that $E_0$ reaches $-1$ in finite time, a feature contradicting \eqref{n2}. Such an approach has been used successfully for the small aspect ratio model \eqref{z0} \cite{FloresMercadoPeleskoSmyth_SIAP07, GuoPanWard_2005} and the stationary version of \eqref{u}-\eqref{bcpsi} \cite{LaurencotWalker_ARMA}, the proof of the latter relying also heavily on the convexity of $u$. But, such a convexity property is not known for the evolution problem \eqref{u}-\eqref{bcpsi} (neither it is for \eqref{z0}) and studying the time evolution of $E_0$ does not seem to work. However, as we shall see below, the study of the time evolution of 
\begin{equation}
E_\alpha(t) := \int_{-1}^1 \zeta_1(x)\ \left( u + \frac{\alpha}{2}\ u^2\right)(t,x)\ \rd x\ , \quad t\in [0,T_m^\varepsilon)\ , \label{n3}
\end{equation}
for a suitable choice of $\alpha\in (0,1)$ leads us to the expected result. Performing that study requires to connect the behavior of $\psi$ to that of $u$ and we devote the next two results to this issue. We first start with an easy consequence of the boundary conditions \eqref{bcpsi}.

\begin{lem}\label{le.n1}
For $t\in [0,T_m^\varepsilon)$ and $p\in [1,\infty)$, we have
\begin{equation}
\frac{4p}{(p+1)^2}\ \int_{-1}^1 \frac{\zeta_1(x)}{1+u(t,x)}\ \rd x \le p\ \int_{\Omega(u(t))} \zeta_1(x)\ \psi(t,x,z)^{p-1}\ |\partial_z \psi(t,x,z)|^2\ \rd (x,z)\ . \label{n4}
\end{equation}
\end{lem}

\begin{proof}[{\bf Proof}]
For $(t,x)\in [0,T_m^\varepsilon)\times I$ and $p\in [1,\infty)$, it follows from \eqref{bcpsi} and the Cauchy-Schwarz inequality that
\begin{align*}
1 = & \psi(t,x,u(t,x))^{(p+1)/2} - \psi(t,x,-1)^{(p+1)/2}\\
 = & \frac{p+1}{2}\ \int_{-1}^{u(t,x)} \psi(t,x,z)^{(p-1)/2}\ \partial_z \psi(t,x,z)\ \rd z \\
\le & \frac{p+1}{2}\ \left( \int_{-1}^{u(t,x)} \psi(t,x,z)^{p-1}\ |\partial_z \psi(t,x,z)|^2\ \rd z \right)^{1/2}\ \sqrt{1+u(t,x)}\ ,
\end{align*}
hence
$$
\frac{4}{(p+1)^2}\ \frac{1}{1+u(t,x)} \le \int_{-1}^{u(t,x)}  \psi(t,x,z)^{p-1}\ |\partial_z \psi(t,x,z)|^2\ \rd z\ .
$$
Owing to the nonnegativity of $\zeta_1$, the estimate \eqref{n4} follows from the above inequality after multiplying both sides by $p \zeta_1(x)$ and integrating over $I$ with respect to $x$.
\end{proof}

The next lemma is a consequence of \eqref{psi}-\eqref{bcpsi}.

\begin{lem}\label{le.n2}
For $t\in [0,T_m^\varepsilon)$ and $p\in [1,\infty)$, we have
\begin{align}
  \int_{-1}^1 &\zeta_1(x)\ \left( 1 + \varepsilon^2\ |\partial_x u(t,x)|^2 \right)\ \partial_z \psi(t,x,u(t,x))\ \rd x \nonumber\\
& = \int_{\Omega(u(t))} \zeta_1(x)\ \left[ p \varepsilon^2\ \psi^{p-1}\ |\partial_x \psi|^2 + p\ \psi^{p-1}\ |\partial_z \psi|^2 + \frac{\mu_1 \varepsilon^2}{p+1}\ \psi^{p+1} \right](t,x,z)\ \rd(x,z) \nonumber\\ 
&  \quad -\ \frac{\mu_1 \varepsilon^2}{(p+1)(p+2)} - \frac{\mu_1\,\varepsilon^2}{p+1}\ \int_{-1}^1 \zeta_1(x)\ u(t,x)\ \rd x\ . \label{n5}
\end{align}
\end{lem}

\begin{proof}[{\bf Proof}]
We multiply \eqref{psi} by $\zeta_1\ \psi^p$ and integrate over $\Omega(u)$. Integrating by parts and using the boundary conditions for $\psi$ and $\zeta_1$ and \eqref{ppsi3}, we obtain
\begin{align*}
0 = & - \int_{\Omega(u)} \left( p \varepsilon^2\ \zeta_1\ \psi^{p-1}\ |\partial_x \psi|^2 + \varepsilon^2\ \partial_x\zeta_1\ \psi^p\ \partial_x\psi + p\ \zeta_1\ \psi^{p-1}\ |\partial_z \psi|^2 \right) \rd (x,z) \\
& + \int_{-1}^1 \zeta_1(x)\ \left[ -\varepsilon^2\ \partial_x u(x)\ \partial_x\psi(x,u(x)) + \partial_z\psi(x,u(x)) \right] \rd x\\
= & \int_{-1}^1 \zeta_1(x)\ \left( 1 + \varepsilon^2\ |\partial_x u(x)|^2 \right)\ \partial_z \psi(x,u(x))\ \rd x \\
& - p\ \int_{\Omega(u)} \zeta_1\ \psi^{p-1} \left[ \varepsilon^2\  |\partial_x \psi|^2 + |\partial_z \psi|^2 \right] \rd (x,z) - \frac{\varepsilon^2}{p+1}\ \int_{\Omega(u)} \partial_x\zeta_1\ \partial_x\left( \psi^{p+1} \right) \rd (x,z)\ .
\end{align*}
Since 
\begin{align*}
- \int_{\Omega(u)} \partial_x\zeta_1\ \partial_x\left( \psi^{p+1} \right) \rd (x,z) = & \int_{\Omega(u)} \partial_x^2\zeta_1\  \psi^{p+1}\ \rd (x,z) \\
& + \left( \partial_x\zeta_1(-1) - \partial_x\zeta_1(1) \right)\ \int_{-1}^0 (1+z)^{p+1}\ \rd z \\
& + \int_{-1}^1 \partial_x\zeta_1(x)\ \partial_x u(x)\ \rd x \\
= & - \mu_1\ \int_{\Omega(u)} \zeta_1\  \psi^{p+1}\ \rd (x,z) + \frac{\mu_1}{p+2} - \int_{-1}^1 \partial_x^2\zeta_1(x)\ u(x)\ \rd x \\
= & - \mu_1\ \int_{\Omega(u)} \zeta_1\  \psi^{p+1}\ \rd (x,z) + \frac{\mu_1}{p+2} + \mu_1 \int_{-1}^1 \zeta_1(x)\ u(x)\ \rd x
\end{align*}
by \eqref{bcpsi} and \eqref{n2b}, we end up with \eqref{n5}.
\end{proof}

\begin{proof}[{\bf Proof of Theorem~\ref{TD} (ii)}]
Let $\alpha\in (0,1)$ to be determined later. We first note that \eqref{n2} implies that the function $E_\alpha$, defined in \eqref{n3}, satisfies
\begin{equation}
-1 \le \frac{\alpha-2}{2} \le E_\alpha(t) \le 0\ ,\qquad t\in [0,T_m^\varepsilon)\ . \label{n4b}
\end{equation}
We next multiply \eqref{u} by $\zeta_1\ (1+\alpha\ u)$, integrate over $I$ and use \eqref{n2b} and \eqref{ppsi3} to obtain
\begin{eqnarray*}
\frac{\rd E_\alpha}{\rd t} & + & \mu_1\ E_\alpha + \alpha\ \int_{-1}^1 \zeta_1\ |\partial_x u|^2\ \rd x = \int_{-1}^1 \zeta_1\ (1+\alpha\ u)\ \left( \partial_t u - \partial_x^2 u \right) \rd x \\
& = & - \lambda\ \int_{-1}^1 \zeta_1(x)\ (1+\alpha\ u(x))\ \left( 1 + \varepsilon^2\ |\partial_x u(x)|^2 \right)\ |\partial_z\psi(x,u(x))|^2\ \rd x\ .
\end{eqnarray*}
Since $\zeta_1(x)\ge 0$ and $1+\alpha\ u(x)\ge 1-\alpha$ by \eqref{n2}, we further obtain
\begin{equation}
\frac{\rd E_\alpha}{\rd t} + \mu_1\ E_\alpha + \alpha\ \int_{-1}^1 \zeta_1\ |\partial_x u|^2\ \rd x \le - \lambda (1-\alpha)\ \mathcal{R} \label{n6}
\end{equation}
with 
$$
\mathcal{R}(t) := \int_{-1}^1 \zeta_1(x)\ \left( 1 + \varepsilon^2\ |\partial_x u(t,x)|^2 \right)\ |\partial_z\psi(t,x,u(t,x))|^2\ \rd x\ , \quad t\in [0,T_m^\varepsilon)\ .
$$
We now look for a lower bound for $\mathcal{R}$. To this end we observe that $\mathcal{R}$ reminds of the left-hand side of \eqref{n5} while \eqref{n4} provides a lower bound of the right-hand side of \eqref{n5}. More precisely, let $\beta>0$ and $p\ge 1$ be two positive real numbers to be determined later. It follows from Young's inequality that
\begin{align*}
\int_{-1}^1 \zeta_1(x)\ & \left( 1 + \varepsilon^2\ |\partial_x u(x)|^2 \right)\ \partial_z\psi(x,u(x))\ \rd x \\ 
\le & \beta\ \mathcal{R} + \frac{1}{4\beta}\ \int_{-1}^1 \zeta_1(x)\ \left( 1 + \varepsilon^2\ |\partial_x u(x)|^2 \right)\rd x\ ,
\end{align*}
that is,
\begin{align*}
\mathcal{R} \ge &\frac{1}{\beta}\ \int_{-1}^1 \zeta_1(x)\ \left( 1 + \varepsilon^2\ |\partial_x u(x)|^2 \right)\ \partial_z\psi(x,u(x))\ \rd x \\ 
& - \frac{1}{4\beta^2}\left(1+\ve^2 \int_{-1}^1 \zeta_1(x)\  |\partial_x u(x)|^2 \rd x\right)\ .
\end{align*}
We now infer from Lemma~\ref{le.n2}, Lemma~\ref{le.n1}, \eqref{n2b}, and the non-positivity of $\zeta_1 u$ that
\begin{align*}
\mathcal{R} \ge & \frac{1}{\beta}\ \left[ p\ \int_{\Omega(u)} \zeta_1\ \psi^{p-1}\ |\partial_z\psi|^2\ \rd(x,z) - \frac{\mu_1 \varepsilon^2}{(p+1)(p+2)} - \frac{\mu_1 \varepsilon^2}{p+1}\ \int_{-1}^1 \zeta_1\ u\ \rd x \right] \\
& - \frac{1}{4\beta^2}\ \left( 1 + \varepsilon^2\ \int_{-1}^1 \zeta_1\ |\partial_x u|^2 \rd x \right)\\
\ge & \frac{1}{\beta}\ \left[ \frac{4p}{(p+1)^2}\ \int_{-1}^1 \frac{\zeta_1}{1+u}\ \rd x - \frac{\mu_1 \varepsilon^2}{(p+1)(p+2)} \right] - \frac{1}{4\beta^2}\ \left( 1 + \varepsilon^2\ \int_{-1}^1 \zeta_1\ |\partial_x u|^2 \rd x \right)\\
\ge & \frac{1}{\beta p}\ \int_{-1}^1 \frac{\zeta_1}{1+u+\alpha\ u^2/2}\ \rd x - \frac{\mu_1 \varepsilon^2}{\beta p^2} - \frac{1}{4\beta^2} - \frac{\varepsilon^2}{4\beta^2}\ \int_{-1}^1 \zeta_1\ |\partial_x u|^2 \rd x \ .
\end{align*}
Finally, since $y\mapsto (1+y)^{-1}$ is convex and $\|\zeta_1\|_{L_1(I)}=1$, we use Jensen's inequality as in \cite{FloresMercadoPeleskoSmyth_SIAP07} and get
$$
\mathcal{R} \ge \frac{1}{\beta p}\ \frac{1}{1+E_\alpha} - \frac{\mu_1 \varepsilon^2}{\beta p^2} - \frac{1}{4\beta^2} - \frac{\varepsilon^2}{4\beta^2}\ \int_{-1}^1 \zeta_1\ |\partial_x u|^2 \rd x\ .
$$
Inserting this estimate in \eqref{n6} and using \eqref{n4b} give 
\begin{align*}
 \frac{\rd E_\alpha}{\rd t} & - \mu_1 + \alpha\ \int_{-1}^1 \zeta_1\ |\partial_x u|^2\ \rd x \\
& \le - \frac{\lambda (1-\alpha)}{\beta p} \left[ \frac{1}{1+E_\alpha} - \frac{\mu_1 \varepsilon^2}{p} - \frac{p}{4\beta} - \frac{p \varepsilon^2}{4\beta}\ \int_{-1}^1 \zeta_1\ |\partial_x u|^2 \rd x \right]\ ,
\end{align*}
whence
\begin{align*}
& \frac{\rd E_\alpha}{\rd t} + \left( \alpha - \frac{\lambda (1-\alpha) \varepsilon^2}{4\beta^2} \right)\ \int_{-1}^1 \zeta_1\ |\partial_x u|^2\ \rd x \le \mu_1 + \frac{\lambda (1-\alpha)}{\beta p} \left[ \frac{\mu_1 \varepsilon^2}{p} + \frac{p}{4\beta} - \frac{1}{1+E_\alpha} \right]\ .
\end{align*}
At this point, the role of the additional parameter $\alpha$ becomes clear as it allows us to control the $\lambda$-dependent term involving $\partial_x u$. We thus choose 
$$
\alpha = \frac{\lambda\varepsilon^2}{4\beta^2+\lambda\varepsilon^2}\in (0,1)\ , \;\;\text{ so that }\;\; \alpha = \frac{\lambda (1-\alpha) \varepsilon^2}{4\beta^2} \ ,
$$
and obtain the following differential inequality for $E_\alpha$:
\begin{equation}
\frac{\rd E_\alpha}{\rd t} \le \mathcal{F}(E_\alpha) := \mu_1 + \frac{4 \lambda \beta}{(4\beta^2+\lambda\varepsilon^2) p} \left[ \frac{\mu_1 \varepsilon^2}{p} + \frac{p}{4\beta} - \frac{1}{1+E_\alpha} \right]\ .\label{n100}
\end{equation}
Since $\mathcal{F}$ is an increasing function on $(-1,\infty)$, it readily follows from the non-positivity of $E_\alpha$ and \eqref{n100} that, if $\mathcal{F}(0)<0$, then $E_\alpha(t)\le E_\alpha(0)\le 0$ and $\rd E_\alpha(t)/\rd t \le \mathcal{F}(0)<0$ for all $t\in [0,T_m^\varepsilon)$. Integrating this differential inequality and using \eqref{n4b}, we conclude that $-1 \le \mathcal{F}(0) t$ for all $t\in [0,T_m^\varepsilon)$ and thus that $T_m^\varepsilon\le -1/\mathcal{F}(0)<\infty$ as claimed.

We are then left with showing that we can find parameters $\beta>0$ and $p\ge 1$ such that $\mathcal{F}(0)<0$ for $\lambda$ large enough. To this end we choose $\beta = \sqrt{\lambda}/2>0$ and $p=1+2\mu_1\varepsilon^2\ge 1$ so that $\alpha=\varepsilon^2/(1+\varepsilon^2)$ and 
$$
\mathcal{F}(0) \le \mu_1 + \frac{2\sqrt{\lambda}}{1+\varepsilon^2}\ \left[ \frac{1}{2} + \frac{1+2\mu_1\varepsilon^2}{2\sqrt{\lambda}} - 1 \right] \le \mu_1 + \frac{\sqrt{\lambda}}{1+\varepsilon^2}\ \left[ \frac{1+2\mu_1\varepsilon^2}{\sqrt{\lambda}} - 1 \right]\ .
$$
Therefore, if $\sqrt{\lambda}>4 \mu_1\ (1+\varepsilon^2)$, we have
$$
\mathcal{F}(0) \le \mu_1 - \frac{\sqrt{\lambda}}{2(1+\varepsilon^2)} < 0\ ,
$$
and the proof of Theorem~\ref{TD}(ii) is  complete.
\end{proof}



}
{
\section{Asymptotic Stability: Proof of Theorem~\ref{TStable}}\label{SectTStable}

Let $q\in (2,\infty)$, $\ve>0$, and $\kappa\in (0,1)$ be fixed. We first prove Theorem~\ref{TStable}(i). Recall that, choosing $\sigma\in (1/2-1/q,1/2)$ so that $W_2^{2\sigma}(I)\hookrightarrow L_q(I)$, Proposition~\ref{L1} states that $g_\ve:S_q(\kappa)\rightarrow L_q(I)$ is analytic. Therefore, since the generator of the heat semigroup $-A:=-A_q\in\mathcal{L}(W_{q,D}^2(I),L_q(I))$ is invertible, we obtain that the mapping
$$
F:\R\times S_q(\kappa)\rightarrow W_{q,D}^2(I)\ ,\quad (\lambda,v)\mapsto v-\lambda A^{-1}g_\ve(v)
$$
is analytic with $F(0,0)=0$ and $D_vF(0,0) =\mathrm{id}_{W_{q,D}^2}$. Now, the Implicit Function Theorem ensures the existence of $\delta>0$ and an analytic function 
$$[\lambda\mapsto U_\lambda]:[0,\delta)\rightarrow W_{q,D}^2(I)$$ such that $F(\lambda,U_\lambda)=0$ for $\lambda\in [0,\delta)$. Denoting the solution to \eqref{psi}-\eqref{bcpsi} with $U_\lambda$ instead of $u$ by $\Psi_\lambda\in W_2^2(\Omega(U_\lambda))$, the pair $(U_\lambda,\Psi_\lambda)$ is for each $\lambda\in (0,\delta)$ the unique steady-state to
\eqref{u}-\eqref{bcpsi} with $U_\lambda$ in $S_q(\kappa)$ and $U_0=0$. Since $U_\lambda$ is convex and satisfies the Dirichlet conditions $U_\lambda(\pm 1)=0$ we clearly have $U_\lambda\le 0$ for $\lambda\in (0,\delta)$. That $U_\lambda$ is even follows from uniqueness and \cite[Thm.1]{LaurencotWalker_ARMA}. This proves Theorem~\ref{TStable} (i).\\

To prove part (ii) of Theorem~\ref{TStable}, we use the Principle of Linearized Stability. For this, we fix $\lambda\in (0,\delta)$ and introduce the linearization of $g_\ve$, 
$$
B_\lambda:=\lambda D_vg_\ve(U_\lambda)\in\mathcal{L}(W_{q,D}^2(I),L_q(I))\ .
$$
Since $\|B_\lambda\|_{\mathcal{L}(W_{q,D}^2(I),L_q(I))}\rightarrow 0$ as $\lambda\rightarrow 0$, it follows from \cite[I.Cor.1.4.3]{LQPP} that $-A-B_\lambda$ is the generator of an analytic semigroup on $L_q(I)$ and there is $\omega_1>0$ such that the complex half plane $[\mathrm{Re}\, z\ge -\omega_1]$ belongs to the resolvent set of $-A-B_\lambda$ provided that $\lambda$ is sufficiently small. 
Now write
$v=u-U_\lambda$ and consider the linearization of \eqref{CP},
\bqn\label{CPL}
\dot{v}+(A+B_\lambda)v=G_\lambda(v)\ ,\quad t>0\ ,\qquad v(0)=v^0\ ,
\eqn
where $G_\lambda\in C^2(O_\lambda,L_q(I))$ is defined on some open zero neighborhood $O_\lambda$ in $W_{q,D}^2(I)$ such that $U_\lambda+O_\lambda\subset S_q(\kappa)$ and given by
$$
G_\lambda(v):=-\lambda\big(g_\ve(U_\lambda+v)-g_\ve(U_\lambda)-D_vg_\ve(U_\lambda)v\big)\ .
$$
Since $-(A+B_\lambda)$ is the generator of an analytic semigroup on $L_q(I)$ with a negative spectral bound as observed above, we may apply \cite[Thm.9.1.1]{Lunardi} and conclude statement (ii) of Theorem~\ref{TStable} by making $\delta>0$ smaller, if necessary.\\

A straightforward consequence of \eqref{est} and \eqref{RR} is:

\begin{cor}\label{RR1}
Under the assumptions of Theorem~\ref{TStable}(ii) there is $R_1>0$ such that
$$
\|\phi_{u(t)}-\phi_{U_\lambda}\|_{W_2^2(\Omega)}\le R_1 e^{-\omega_0 t}\| u^0-U_\lambda\|_{W_{q,D}^2(I)}\ ,\quad t\ge 0\ ,
$$
where $\phi_v$ is defined in Proposition~\ref{L1}.
\end{cor}
}

\section{Small Aspect Ratio Limit: Proof of Theorem~\ref{Bq}}\label{Sec4} 

We shall now prove Theorem~\ref{Bq}. {Fix $\lambda>0$, $q\in (2,\infty)$, $\kappa\in (0,1)$, and let $u^0\in S_q(\kappa)$ with $u^0(x)\le 0$ for $x\in I$. For $\varepsilon>0$ we denote the unique solution to \eqref{u}-\eqref{bcpsi} by $(u_\varepsilon,\psi_\varepsilon)$ which is defined on the maximal interval of existence $[0,T_m^\varepsilon)$. 
In the following, $(K_i)_{i\ge 1}$ and $K$ denote positive constants depending only on $q$ and $\kappa$, but not on $\ve>0$ sufficiently small.

Set $\kappa_0 := \kappa/(2M)< \kappa$, where $M\ge 1$ is the constant defined in \eqref{9}. Owing to the continuity properties of $u_\varepsilon$, we have
\begin{equation}
\tau^\varepsilon := \sup{\left\{ t\in [0,T_m^\varepsilon)\ :\ u_\varepsilon(s)\in\overline{S}_q(\kappa_0) \;\;\text{ for all }\;\; s\in [0,t] \right\}} > 0\ . \label{pam2}
\end{equation}
Thanks to the continuity of the embedding of $W_q^2(I)$ in $W_\infty^1(I)$, there is a positive constant $K_1$ such that, for all $\varepsilon>0$,
\begin{align}
-1 + \kappa_0 \le u_\varepsilon(t,x) & \le 0\,, \qquad (t,x)\in [0,\tau^\varepsilon]\times [-1,1]\,, \label{z1} \\
\|u_\varepsilon(t)\|_{W_q^2(I)} + \|u_\varepsilon(t)\|_{W_\infty^1(I)} & \le K_1\,, \qquad t\in [0,\tau^\varepsilon]\,. \label{z2}
\end{align}
As a consequence of \eqref{z2} there is $\varepsilon_0>0$ depending only $q$ and $\kappa$} such that 
\begin{equation}
\varepsilon_0^2\ \left\| \partial_x u_\varepsilon(t) \right\|_{L_\infty(I)}^2 \le \frac{1}{2}\,, \qquad (t,\varepsilon)\in [0,\tau^\varepsilon]\times (0,\varepsilon_0]\,. \label{z3}
\end{equation}
For $\e\in (0,\varepsilon_0)$, we set 
$$
\phi_\e(t) := \phi_{u_\varepsilon(t)}=\psi_\varepsilon(t) \circ T_{\ue(t)}^{-1}\ ,\qquad t\in [0,\tau^\varepsilon]\ ,
$$ 
with $T_{\ue(t)}^{-1}$ given by \eqref{Tuu} and 
$$
\Phi_\varepsilon(t,x,\eta) := \phi_\varepsilon(t,x,\eta)-\eta\ ,\qquad (t,x,\eta)\in [0,\tau^\varepsilon]\times\overline{\Omega}\,.
$$ 
The cornerstone of the proof of Theorem~\ref{Bq} is to derive appropriate estimates on $\Phi_\varepsilon$, showing that it converges to zero as $\varepsilon\to 0$. To this end, we further develop and improve the analysis performed in \cite[Section~3]{LaurencotWalker_ARMA} and establish the following bounds:

\begin{lem}\label{le:z1}
There exists a positive constant $K_2$ such that, for $\varepsilon\in (0,\varepsilon_0)$ and $t\in [0,\tau^\varepsilon]$, 
\begin{align}
\left\| \partial_x\Phie(t) \right\|_{L_2(\Omega)} + \frac{1}{\e}\ \left( \left\| \Phie(t) \right\|_{L_2(\Omega)} + \left\| \partial_\eta\Phie(t) \right\|_{L_2(\Omega)} \right) & \le K_2\,, \label{z41} \\
\frac{1}{\varepsilon} \left\| \partial_x\partial_\eta\Phie(t) \right\|_{L_2(\Omega)} + \frac{1}{\varepsilon^2}\ \left\| \partial_\eta^2\Phie(t)\right\|_{L_2(\Omega)} & \le K_2\,, \label{z42} \\
\frac{1}{\varepsilon}\ \left\| \partial_\eta{ \Phie(t,\cdot,1)} \right\|_{W_2^{1/2}(I)} & \le K_2\,. \label{z5}
\end{align}
\end{lem}

\begin{proof}[{\bf Proof}]
Fix $\varepsilon\in (0,\varepsilon_0)$ and $t\in [0,\tau^\varepsilon]$. It first follows from \eqref{ppsi1} that 
\begin{equation}
\left\| \Phie(t) \right\|_{L_\infty(\Omega)} \le 1\,, \label{z6}
\end{equation}
while \eqref{z1} and \eqref{z2} entail that the function 
$$
f_\varepsilon(t,x,\eta) := f_{u_\varepsilon(t)}(x,\eta) = \varepsilon^2\ \eta\ \left[ 2\ \left( \frac{\partial_x u_\varepsilon}{1+u_\varepsilon} \right)^2 - \frac{\partial_x^2 u_\varepsilon}{1+u_\varepsilon} \right](t,x)\,, \qquad (t,x,\eta)\in[0,\tau^\ve]\times\Omega\,,
$$ 
satisfies
\begin{align*}
\|f_\varepsilon(t)\|_{L_q(\Omega)} \le &\,  \varepsilon^2\ \left[ \frac{2}{\kappa_0^2}\ \left\| \partial_x u_\varepsilon(t) \right\|_{L_\infty(I)}\ \left\| \partial_x u_\varepsilon(t) \right\|_{L_q(I)} + \frac{1}{\kappa_0}\ \left\| \partial_x^2 u_\varepsilon(t) \right\|_{L_q(I)} \right] \\
\le & \left( \frac{2 K_1^2}{\kappa_0^2} + \frac{K_1}{\kappa_0} \right)\ \varepsilon^2\,.
\end{align*}
Therefore, by H\"older's inequality,
\begin{equation}
\|f_\varepsilon(t)\|_{L_p(\Omega)} \le 2^{(q-p)/qp}\ \|f_\varepsilon(t)\|_{L_q(\Omega)} \le K_3\ \varepsilon^2 \label{z7}
\end{equation}
for $p\in [1,q]$. From now on, the time $t$ plays no particular role anymore and is thus omitted in the notation. We multiply \eqref{23a} by $\Phie$, integrate over $\Omega$, and proceed as in \cite[Lemma~11]{LaurencotWalker_ARMA} to obtain
\begin{align*}
\int_\Omega f_\e\ \left( 1 - \partial_\eta\Phie \right)\ \Phie\ \rd (x,\eta) = &\, \e^2\ \int_\Omega \left( \partial_x\Phie - \eta\ \frac{\partial_x u_\e}{1+u_\e}\ \partial_\eta\Phie \right)^2\ \rd (x,\eta) \\
& + \int_\Omega \frac{|\partial_\eta\Phie|^2}{(1+u_\e)^2}\ \rd (x,\eta)
\end{align*}
To estimate the right-hand side of the above identity from below, we use the elementary inequality $(r-s)^2 \ge (r^2/2)-s^2$, \eqref{z1}, and \eqref{z3} to obtain
\begin{align*}
\int_\Omega f_\e\ \left( 1 - \partial_\eta\Phie \right)\ \Phie\ \rd (x,\eta) \ge & \frac{\varepsilon^2}{2}\ \left\| \partial_x\Phie \right\|_{L_2(\Omega)}^2 - \e^2\ \left\| \partial_x u_\varepsilon \right\|_{L_\infty(I)}^2\ \int_\Omega  \frac{\left| \partial_\eta\Phie \right|^2}{(1+u_\e)^2}\ \ \rd (x,\eta) \\
& + \int_\Omega \frac{|\partial_\eta\Phie|^2}{(1+u_\e)^2}\ \rd (x,\eta) \\
\ge & \frac{\varepsilon^2}{2}\ \left\| \partial_x\Phie \right\|_{L_2(\Omega)}^2 + \frac{1}{2}\ \int_\Omega \frac{|\partial_\eta\Phie|^2}{(1+u_\e)^2}\ \rd (x,\eta) \\
\ge & \frac{\varepsilon^2}{2}\ \left\| \partial_x\Phie \right\|_{L_2(\Omega)}^2 + \frac{1}{2}\ \left\| \partial_\eta\Phie \right\|_{L_2(\Omega)}^2\,.
\end{align*}
Next, thanks to \eqref{z6}, \eqref{z7}, and H\"older's inequality, we can estimate the left-hand side of the above inequality and obtain
\begin{align*}
\varepsilon^2\ \left\| \partial_x\Phie \right\|_{L_2(\Omega)}^2 +  \left\| \partial_\eta\Phie \right\|_{L_2(\Omega)}^2 & \le 2\ \left\| f_\varepsilon \right\|_{L_2(\Omega)}\ \left\| 1 - \partial_\eta\Phie \right\|_{L_2(\Omega)}\ \left\| \Phie \right\|_{L_\infty(\Omega)} \\
& \le 2K_3\varepsilon^2\ \left( 1 + \left\| \partial_\eta\Phie \right\|_{L_2(\Omega)} \right) \\
& \le 2K_3\varepsilon^2 + \frac{1}{2}\ \left\| \partial_\eta\Phie \right\|_{L_2(\Omega)}^2 + 2K_3^2\varepsilon^4\,,
\end{align*}
whence
\begin{equation}
\varepsilon^2\ \left\| \partial_x\Phie \right\|_{L_2(\Omega)}^2 +  \left\| \partial_\eta\Phie \right\|_{L_2(\Omega)}^2 \le K_4\ \varepsilon^2\,. \label{z8}
\end{equation}
Since $\Phie(x,1)=0$ for $x\in I$, we have $\|\Phie\|_{L^2(\Omega)}\le \sqrt{2}\ \left\| \partial_\eta\Phie \right\|_{L_2(\Omega)}$ and \eqref{z41} readily follows from this inequality and \eqref{z8}.

We next establish \eqref{z42}. For that purpose, we set $\zeta_\e := \partial_\eta^2\Phie$, $\omega_\e := \partial_x\partial_\eta\Phie$, and multiply \eqref{23} by $\zeta_\varepsilon$. After integrating over $\Omega$, we proceed as in \cite[Lemma~11]{LaurencotWalker_ARMA} with the help of \cite[Lem.~4.3.1.2 $\&$~4.3.1.3]{Grisvard} to deduce that
\begin{align*}
\int_\Omega f_\e\ \left( 1 - \partial_\eta\Phie \right)\ \zeta_\e\ \rd (x,\eta) = \int_\Omega \left[ \frac{\zeta_\e^2}{(1+u_\e)^2} + \e^2\ \left( \omega_\e - \eta\ \frac{\partial_x u_\e}{1+u_\e}\ \zeta_\e \right)^2 \right]\ \rd (x,\eta)\,.
\end{align*}
Using once more the inequality $(r-s)^2 \ge (r^2/2)-s^2$ and \eqref{z3} to estimate the right-hand side of the above inequality from below, we find
\begin{align*}
\int_\Omega f_\e\ \left( 1 - \partial_\eta\Phie \right)\ \zeta_\e\ \rd (x,\eta) \ge & \int_\Omega \left[ \frac{\zeta_\e^2}{(1+u_\e)^2} + \frac{\e^2}{2}\ \omega_\e^2 - \varepsilon^2 \eta^2\ \frac{\left| \partial_x u_\e \right|^2}{(1+u_\e)^2}\ \zeta_\e^2 \right]\ \rd (x,\eta) \\
\ge & \int_\Omega \left[ \frac{1}{2}\ \frac{\zeta_\e^2}{(1+u_\e)^2} + \frac{\e^2}{2}\ \omega_\e^2 \right]\ \rd (x,\eta) \\
\ge & \frac{1}{2}\ \left( \left\| \zeta_\varepsilon \right\|_{L_2(\Omega)}^2 + \varepsilon^2\ \left\| \omega_\varepsilon \right\|_{L_2(\Omega)}^2 \right)\,.
\end{align*}
Introducing $Q_\varepsilon := \sqrt{\left\| \zeta_\varepsilon \right\|_{L_2(\Omega)}^2 + \varepsilon^2\ \left\| \omega_\varepsilon \right\|_{L_2(\Omega)}^2}$, we infer from H\"older's inequality, \eqref{z7}, and the previous inequality that (recall that $q>2$)
\begin{align*}
Q_\varepsilon^2 & \le 2\ \left\| f_\varepsilon \right\|_{L_q(\Omega)}\ \left\| 1 - \partial_\eta\Phie \right\|_{L_{2q/(q-2)}(\Omega)}\ \left\| \zeta_\varepsilon \right\|_{L_2(\Omega)} \\
& \le K\varepsilon^2\ \left( 1 + \left\| \partial_\eta\Phie \right\|_{L_{2q/(q-2)}(\Omega)} \right)\ Q_\varepsilon\,,
\end{align*}
that is,
\begin{equation}
Q_\varepsilon \le K\varepsilon^2\ \left( 1 + \left\| \partial_\eta\Phie \right\|_{L_{2q/(q-2)}(\Omega)} \right)\,. \label{z9}
\end{equation}
At this point, we infer from the Gagliardo-Nirenberg inequality \cite{Nirenberg} and \eqref{z8},that
\begin{align*}
\left\| \partial_\eta\Phie \right\|_{L_{2q/(q-2)}(\Omega)} & \le K\ \left\| \partial_\eta\Phie \right\|_{W_2^1(\Omega)}^{2/q}\ \left\| \partial_\eta\Phie \right\|_{L_2(\Omega)}^{(q-2)/q} \\
& \le K\ \left( \left\| \partial_\eta\Phie \right\|_{L_2(\Omega)}^2 + \left\| \omega_\varepsilon \right\|_{L_2(\Omega)}^2 + \left\| \zeta_\varepsilon \right\|_{L_2(\Omega)}^2 \right)^{1/q}\ \varepsilon^{(q-2)/q} \\
& \le K \varepsilon^{(q-4)/q}\ \left( \varepsilon^4 + \varepsilon^2\ \left\| \omega_\varepsilon \right\|_{L_2(\Omega)}^2 + \varepsilon^2\ \left\| \zeta_\varepsilon \right\|_{L_2(\Omega)}^2 \right)^{1/q} \\
& \le K \varepsilon^{(q-4)/q}\ \left( \varepsilon^{4/q} + Q_\varepsilon^{2/q} \right) \\
& \le K\ \left( \varepsilon + \varepsilon^{(q-4)/q}\  Q_\varepsilon^{2/q} \right)\,.
\end{align*}
Inserting this estimate in \eqref{z9} leads us to
\begin{align*}
Q_\varepsilon & \le K\varepsilon^2\ \left( 1 + \varepsilon + \varepsilon^{(q-4)/q}\  Q_\varepsilon^{2/q} \right) \le K\varepsilon^2 + K \varepsilon^{(3q-4)/q}\ Q_\varepsilon^{2/q} \\
& \le K\varepsilon^2 + \frac{2}{q}\ Q_\varepsilon + 
K \varepsilon^{(3q-4)/(q-2)}\,,
\end{align*}
whence
$$
Q_\varepsilon \le K\varepsilon^2\ \left( 1 + \varepsilon^{q/(q-2)} \right) \le K\varepsilon^2\,,
$$
and the proof of \eqref{z42} is complete. 

As a consequence of \eqref{z41} and \eqref{z42}, we have $\left\| \partial_\eta\Phie \right\|_{W_2^1(\Omega)}\le K\varepsilon$ and the properties of the trace operator readily give \eqref{z5}, see \cite[Thm.~1.5.1.3]{Grisvard}. 
\end{proof}

{ A first consequence of Lemma~\ref{le:z1} is that $\tau^\varepsilon$ (and thus also $T_m^\varepsilon$) does not collapse to zero as $\varepsilon\to 0$, so that the solutions $(u_\varepsilon,\psi_\varepsilon)_{\varepsilon\in (0,\varepsilon_0)}$ to \eqref{u}-\eqref{bcpsi} have a common interval of existence.
\begin{lem}\label{le.pam1}
\begin{itemize}
\item[(i)] There is $\tau>0$ depending only on $q$, $\lambda$, and $\kappa$ such that $\tau^\varepsilon\ge \tau$ for all $\varepsilon\in (0,\varepsilon_0)$.
\item[(ii)] There is $\Lambda:=\Lambda(\kappa)>0$ such that $\tau^\varepsilon=T_m^\ve=\infty$ for all $\ve\in (0,\ve_0)$ provided $\lambda\in (0,\Lambda)$.
\end{itemize}
\end{lem}
\begin{proof}[{\bf Proof}]
Owing to \eqref{z1}, \eqref{z2} and \eqref{z5}, we may argue as at the end of the proof of Proposition~\ref{L1} to conclude that, fixing $2\sigma\in (1/2-1/q,1/2)$, there is $K_5>0$ such that
\begin{equation}
\left\| g_\varepsilon(u_\e(t)) \right\|_{W_2^{2\sigma}(I)} \le K_5\ , \qquad t\in [0,\tau^\varepsilon]\ . \label{pam3}
\end{equation}
As in the proof of \eqref{11} and \eqref{13}, we infer from \eqref{9}, \eqref{pam3}, the fact that $u^0\in S_q(\kappa)$, and the Variation-of-Constant formula that, for $t\in [0,\tau^\ve]$, 
\begin{align}
\|u_\varepsilon(t)\|_{W_{q,D}^{2}(I)} & \le  M\ \|u^0\|_{W_{q,D}^2(I)} + \lambda\ M\ \int_0^t e^{-\omega(t-s)}\ (t-s)^{\sigma-1-\frac{1}{2}(\frac{1}{2}-\frac{1}{q})}\ \|g_\varepsilon(u_\varepsilon(s))\|_{W_{2,D}^{2\sigma}(I)}\,\rd s \nonumber \\
&\le  \frac{M}{\kappa}+\lambda\ M\ K_5\ \mathcal{I}(t)\ , \label{pam4}
\end{align}
and, using in addition the embedding of $W_{q,D}^2(I)$ in $L_\infty(I)$ with constant 2,
\begin{align}
\|u_\varepsilon(t)\|_{L_\infty(I)} \le & 1 - \kappa + 2\,\lambda\,  \int_0^t \left\| e^{-(t-s)A}\ g_\varepsilon(u_\varepsilon(s)) \right\|_{W_{q,D}^2(I)}\, \rd s \nonumber \\
\le & 1-\kappa +  2\,\lambda\,  M \,  \int_0^t e^{-\omega (t-s)}\ (t-s)^{\sigma-1-\frac{1}{2}(\frac{1}{2}-\frac{1}{q})}\ \|g_\varepsilon(u_\varepsilon(s))\|_{W_{2,D}^{2\sigma}(I)}\,\rd s \nonumber\\
\le & 1-\kappa +  2\,\lambda\, M\, K_5\ \mathcal{I}(t)\ , \label{pam5}
\end{align}
where $\mathcal{I}(t)$ is defined in \eqref{volvic}. Since $\mathcal{I}(t)\to 0$ as $t\to 0$, there exists $\tau>0$ which depends only on $q$ and $\kappa$ such that
\bqn\label{zui}
\mathcal{I}(t) < \frac{1}{\lambda \kappa K_5} \;\;\text{ and }\;\; \mathcal{I}(t) < \frac{(2M-1)\kappa}{4\lambda M^2 K_5} \;\;\text{ for all }\;\; t\in [0,\tau]\ .
\eqn
Thanks to this choice, we readily deduce from \eqref{pam4} and \eqref{pam5} that $\|u_\varepsilon(t)\|_{W_{q,D}^{2}(I)}  \le 1/\kappa_0$ and $u_\varepsilon(t)\le 1- \kappa_0$ for all $t\in [0,\tau]\cap [0,\tau^\varepsilon]$. In other words, $u_\varepsilon(t)\in\overline{S}_q(\kappa_0)$ for all $t\in [0,\tau]\cap [0,\tau^\varepsilon]$ and the definition of $\tau^\varepsilon$ implies that $\tau^\varepsilon\ge \tau$. Finally, setting
$$
\Lambda(\kappa):=\mathrm{min}\left\{\frac{1}{\kappa K_5 \mathcal{I}(\infty)}\,,\,\frac{(2M-1)\kappa}{4M^2K_5 \mathcal{I}(\infty)}\right\}\ ,
$$
and taking $\lambda\in (0,\Lambda(\kappa))$,
it follows that \eqref{zui} holds for any $\tau>0$ which entails that $\tau^\ve\ge \tau$ for any $\tau>0$.
\end{proof}
}

\begin{proof}[{\bf Proof of Theorem~\ref{Bq}}]
A straightforward consequence of \eqref{z5} and the continuous embedding of $W_2^{1/2}(I)$ in $L_{2q}(I)$ is that
\begin{equation}
\lim_{\varepsilon\to 0}\ \sup_{t\in [0,\tau]} \left\| \left| \partial_\eta\phi_\varepsilon(t,\cdot,1) \right|^2 - 1 \right\|_{L_q(I)} = 0\,. \label{z91}
\end{equation}
Since
\begin{equation}
\partial_t u_\e -\partial_x^2 u_\e = - \lambda \, g_\varepsilon(u_\e(t))\,,\qquad x\in I\,, \quad t\in [0,\tau]\ ,\label{z10}
\end{equation}
with $g_\varepsilon$ defined in Proposition~\ref{L1} and $\tau$ in Lemma~\ref{le.pam1}, it readily follows from \eqref{z1}, \eqref{z2}, Lemma~\ref{le:z1}, and the continuous embedding of $W_2^{1/2}(I)$ in $L_{2q}(I)$ that, for $t\in [0,\tau]$,
\begin{align*}
\left\| \partial_t u_\varepsilon(t) \right\|_{L_q(I)} \le & \left\| \partial_x^2 u_\varepsilon(t) \right\|_{L_q(I)} + \lambda\ \left\| \frac{1+\varepsilon^2 \left| \partial_x u_\varepsilon(t)\right|^2}{(1+u_\varepsilon(t))^2} \right\|_{L_\infty(I)}\ \left\| \partial_\eta \phi_\varepsilon(t,\cdot,1) \right\|_{L_{2q}(I)}^2 \\
& \le K_1 + \frac{9\lambda K}{8 \kappa_0^2}\ \left\| 1 + \partial_\eta \Phi_\varepsilon(t,\cdot,1) \right\|_{W_2^{1/2}(I)}^2 \\
& \le K\, (1+\lambda)\,.
\end{align*} 
Consequently, the family $(u_\ve)_{\ve\in (0,\varepsilon_0)}$ is bounded in $C^1([0,\tau],L_q(I))\cap C([0,\tau],W_q^2(I))$ and thus, given $\theta\in ((q+1)/2q,1)$, we may extract a sequence $(\e_k)_{k\ge 1}$ of positive real numbers, $\varepsilon_k \to 0$, such that 
\begin{equation}\label{z11}
u_{\e_k} \longrightarrow u_0\quad\text{in} \quad C^{1-\theta}([0,\tau],W_q^{2\theta}(I))
\end{equation}
for some function $u_0\in C^{1-\theta}([0,\tau],W_q^{2\theta}(I))$. Since $\theta>(q+1)/2q$, we deduce from \eqref{z11} and the continuous embedding of $W_q^{2\theta}(I)$ in $W_\infty^1(I)$ that
\begin{equation}\label{z12}
u_{\e_k} \longrightarrow u_0\quad\text{in} \quad C([0,\tau],W_\infty^{1}(I))\,.
\end{equation}
A first consequence of \eqref{z1} and \eqref{z12} is that $-1+\kappa_0\le u_0(t,x)\le 0$ for all $(t,x)\in [0,\tau]\times [-1,1]$. Next, combining \eqref{z91} and \eqref{z12} ensures that
\begin{equation*}
g_{\e_k}(u_{\e_k}(t))=\frac{1 + \e_k^2\left( \partial_x u_{\e_k}(t) \right)^2}{\left( 1+u_{\e_k}(t) \right)^2}\ \left| \partial_\eta \phi_{\e_k}(t,.,1) \right|^2 \longrightarrow \frac{1}{(1+u_0(t))^2} \quad\text{ in }\quad C([0,\tau],L_q(I))\,.
\end{equation*} 
Recalling \eqref{z10}, classical stability properties of the linear heat equation entails that $u_0$ is a solution to the small aspect ratio equation \eqref{z0} and it is clearly the unique solution to \eqref{z0} which belongs to $\overline{S}_q(\kappa_0)$ for all $t\in [0,\tau]$. This implies in particular that not only a subsequence but the whole family $(u_\ve)_{\ve\in (0,\varepsilon_0)}$ converges towards $u_0$ in $C^{1-\theta}([0,\tau],W_q^{2\theta}(I))$, $\theta\in (0,1)$, as $\varepsilon\to 0$. Finally, the remaining assertion \eqref{z00} follows easily from Lemma~\ref{le:z1} as shown in the proof of \cite[Thm.~2]{LaurencotWalker_ARMA}. This completes the proof of Theorem~\ref{Bq}.
\end{proof}

\section*{Acknowledgments}

Part of this research was done while J.E. and Ch.W. were visiting the Institut de Math\'{e}ma\-ti\-ques de Toulouse, Universit\'{e} Paul Sabatier. The financial support and kind hospitality is gratefully acknowledged.



\end{document}